\documentclass[notitlepage,11pt,reqno]{amsart}
\usepackage{color}
\usepackage{amssymb,nicefrac,bm,upgreek,mathtools,verbatim,enumerate,ulem}
\usepackage[mathscr]{eucal}
\usepackage{graphicx}
\usepackage{epsf,times}

\definecolor{dmagenta}{rgb}{.4,.1,.5}
\definecolor{007}{rgb}{.0,.0,.7}
\definecolor{dred}{rgb}{.5,.0,.0}
\definecolor{dgreen}{rgb}{.0,.5,.0}
\definecolor{dblue}{rgb}{.0,.0,.5}

\definecolor{violet}{rgb}{.3,.0,.9}
\definecolor{orange}{cmyk}{0,.5,.1,.0}
\definecolor{dcyan}{cmyk}{.5,.0,.0,.0}
\definecolor{dyellow}{cmyk}{.0,.0,.5,.0}
\definecolor{cm}{cmyk}{1,.0,.0,.0}
\numberwithin{equation}{section}
\allowdisplaybreaks

\newtheorem{theorem}{Theorem}[section]
\newtheorem{lemma}{Lemma}[section]
\newtheorem{proposition}{Proposition}[section]

\theoremstyle{definition}
\newtheorem{definition}{Definition}[section]
\newtheorem{example}{Example}[section]

\theoremstyle{remark}
\newtheorem{remark}{Remark}[section]

\newcommand{\grad}{\nabla}

\newcommand{\R}{\mathbb{R}}
\newcommand{\bS}{\mathbb{S}}

\newcommand{\RN}{\mathbb{R}^N}

\newcommand{\sB}{\mathscr{B}}
\newcommand{\cC}{\mathcal{C}}

\newcommand{\sE}{\mathscr{E}}

\newcommand{\cO}{\mathcal{O}}

\newcommand{\sL}{\mathscr{L}}

\newcommand{\Ind}{\mathbb{I}}

\newcommand{\bvnorm}[1]{[\kern-0.45ex[\kern0.1ex #1 \kern0.1ex]\kern-0.45ex]}

\newcommand{\abs}[1]{\lvert#1\rvert}
\newcommand{\norm}[1]{\lVert#1\rVert}
\newcommand{\sgn}{\mathrm{sgn}}

\newcommand{\df}{:=}
\newcommand{\infdel}{\Delta_\infty}
\newcommand{\ginfdel}{\Delta^\gamma_\infty}
\newcommand{\plam}{\lambda^{\prime}_{1}}
\newcommand{\pplam}{\lambda^{\prime\prime}_{1}}

\DeclareMathOperator*{\argmax}{arg\,max}

\DeclareMathOperator*{\supp}{support}

\DeclareMathOperator{\dist}{dist}

\begin{document}

\title[Principal eigenvalue problem for infinity Laplacian in $\RN$]
{Harnack inequality and principal eigentheory for general  infinity Laplacian  operators with gradient in $\RN$ and applications}

\author{Anup Biswas}
\address{Indian Institute of Science Education and Research, Dr.\ Homi Bhabha Road, Pashan, Pune 411008}
\email{anup@iiserpune.ac.in}

\author{Hoang-Hung Vo}
\address{Faculty  of  Mathematics  and  Applications,  Saigon  University,  273  An  Duong  Vuong  st.,  Ward  3,Dist.5, Ho Chi Minh City, Viet Nam}
\email{vhhung@sgu.edu.vn}

\date{}

\begin{abstract}
Under the lack of variational structure and nondegeneracy, we investigate three notions of \textit{generalized principal eigenvalue}  for a general infinity Laplacian operator with gradient and homogeneous term. A Harnack inequality is proved to support our analysis. This is a continuation of our first work \cite{BV20} and
a contribution in the development of the theory of \textit{generalized principal eigenvalue} beside the works \cite{BNV,BR06,BR15,BCPR,NV19}. We use these notions to characterize the validity of maximum principle  and study the existence, nonexistence and uniqueness of positive solutions of Fisher-KPP type equations in the whole space. The sliding method is intrinsically improved for infinity Laplacian to solve the problem. The results are related to the Liouville type results, which will be meticulously explained.
\end{abstract}

\maketitle



\section{Introduction and Main results}
The notion of generalized principal eigenvalue was first introduced in the celebrated work of 
Berestycki-Nirenberg-Varadhan \cite{BNV} and it became an important basic of the theory of partial differential equations because of its  usefulness in the study of the existence/nonexistence of positive solution, Liouville type results and maximum principle. It also has significant impact in probability theory, especially in the theory of large deviations. The infinity Laplacian first appeared in an interesting 
series of works of G. Aronsson \cite{AG1,AG2,AG3} while studying {\it absolute minimizer} in a domain of 
$\RN$. Later it also 
found applications in image processing \cite{CMS} and the {\it tug-of-war} game of Peres-Schramm-Sheffield-Wilson \cite{PSSW}. For further insight to this topic we mention the expositions \cite{ACJ,C08,L16}. The purpose of this paper is twofold. First, we are interested in the study of eigentheory and qualitative properties of generalized principal eigenvalues for general infinity Laplacian operators with a gradient and homogeneous term. Second, we use these notions  to characterize the validity of maximum principle and the existence/nonexistence of positive solutions of Fisher-KPP type equation in $\RN$.

Throughout the paper, given $\gamma\in [0, 2]$, we define the operator $\sL$ as follow
$$\sL u = \ginfdel u + H(x, \grad u) =  \frac{1}{\abs{\grad u}^\gamma} \sum_{i, j=1}^N \partial_{x_i} u\, \partial_{x_i x_j} u\, \partial_{x_j} u + H(x, \grad u),$$
where $H(x, q)$ is a continuous Hamiltonian satisfying the following properties:
\begin{itemize}
\item[(a)] $H$ is positively $(3-\gamma)$-homogeneous in $q$, that is, $H(x, tq)=t^{3-\gamma} H(x, q)$ for all $t>0$.
\item[(b)] For any compact set $K$, there exists $C_K$ satisfying $|H(x, q)|\leq C_K |q|^{3-\gamma}$ for all $q\in\RN$ and $x\in K$.
\item[(c)] For any compact set $K$, we have 
\begin{align}\label{EH}
|H(x, q)-H(y, q)|\leq \omega_K(|x-y|) (1+ |q|^{3-\gamma}),
\end{align}
for $x, y\in K$ and $\omega_K$ is function of modulus of continuity, that is, $\omega_K:[0,\infty)\to [0, \infty)$
is a continuous function satisfying $\omega_K(0)=0$.
\end{itemize}
Some typical examples of $H$ are $|b(x)\cdot q|^{3-\gamma}$,  $b(x)\cdot q |q|^{2-\gamma}$ etc. Also,
we note that $\gamma=0$ corresponds to the infinity Laplacian whereas $\gamma=2$ corresponds to
the normalized infinity Laplacian. Infinity Laplace equation with a gradient term appears in the study
of certain tug-of-war games, see \cite{LNR13}. There are quite a
few recent works that also consider infinite Laplace equation
combined with a gradient term, see \cite{ASS11, SP, PV12, PV13}.
This gives us the motivation to consider a general operator $\sL$
for our study.
In this paper, unless otherwise mentioned, we shall always employ the concept of viscosity solution in all equations and inequations,  the precise definition is given in  Section~\ref{S-Har}. Let $c:\RN\to\R$ be a given continuous function. Given a smooth bounded domain $\cO$, following \cite{BD06,BV20,BM13,PJ07}, we define  the generalized principal eigenvalue of 
$\sL+c$ as 
\begin{equation}\label{E1.1}
\lambda_\cO(\sL+c)=\sup\{\lambda\in\R\; :\; \exists\; \psi\in\cC(\bar\cO)\;
\text{satisfying}\; \sL\psi+ (c(x)+\lambda)\psi^{3-\gamma}\leq 0\; \text{in}\; \cO, \;\text{and}\; \inf_{\cO}\psi>0\}.
\end{equation}
{ The above notion of generalized principal eigenvalue dates back to the 
works of Nussbaum-Pinchover \cite{NP92}, Berestycki-Nirenberg-Varadhan \cite{BNV}.}
It is well known  from  \cite{BNV} that the maximum principle can be directly characterized by the positivity of the principal eigenvalue. In that spirit, Berestycki et.\ al.\ \cite{BCPR} used the following quantity to characterize the maximum principle for degenerate operators :
\begin{align*}
\upmu_\cO(\sL+c)=\sup\{\lambda\in\R\; &:\; \exists\;\cO_1\Supset\cO,\; \psi\in\cC(\cO_1)\;
\text{satisfying}\; \sL\psi+ (c(x)+\lambda)\psi^{3-\gamma}\leq 0\; \text{in}\; \cO_1,
\\
&\qquad \;\text{and}\; \psi>0\; \text{in}\; \cO_1\}.
\end{align*}
We have also proved in our previous work \cite[Lemma~3.6]{BV20} that $\lambda_\cO=\upmu_{\cO}$ and $\sL +c$ satisfies a weak maximum principle 
if $\lambda_{\cO}(\sL+c)>0$  \cite[Lemma~3.2]{BV20}. In particular, we have
the following characterization for the principal eigenvalue defined in (\ref{E1.1}) (see \cite[Corollary~3.1]{BV20})
\begin{equation}\label{E1.2}
\lambda_\cO(\sL+c)=\inf\{\lambda\in\R\; :\; \exists\; \psi\in\tilde\cC^+_0(\cO)\;
\text{satisfying}\; \sL\psi+ (c(x)+\lambda)\psi^{3-\gamma}\geq 0\; \text{in}\; \cO\},
\end{equation}
where $\tilde\cC^+_0(\cO)$ denotes the set of all nontrivial, non-negative continuous functions on $\bar\cO$ vanishing on $\partial\cO$. It is worth pointing out that the collection of admissible test functions
in the definition of $\lambda_\cO$ in (\ref{E1.1}) is strictly smaller than those considered in \cite[expression~(1.10)]{BNV}.
More precisely, an analogous quantity of \cite{BNV} should read as follows
\begin{equation}\label{5.13.1}
\lambda_1(\cO, \sL+c)=\sup\{\lambda\in\R\; :\; \exists\;\; \psi\in\cC(\bar\cO)\;
\text{satisfying}\;\psi>0\;\text{and}\; \sL\psi+ (c(x)+\lambda)\psi^{3-\gamma}\leq 0\; \text{in}\; \cO\}.
\end{equation}
From \cite[Theorem~3.1]{BV20}, if $\cO$ is a smooth bounded domain in $\R^N$,  there exists a positive $\varphi\in\tilde\cC^+_0(\cO)$ satisfying 
\begin{equation}\label{Dirichlet}
\left\lbrace
\begin{array}{ll}
\sL\varphi + (c(x) + \lambda_\cO)\varphi^{3-\gamma}=0& \textrm{in $\cO$},
\\
\hspace{9em}\varphi=0& \textrm{on $\partial\cO$},
\end{array}
\right.
\end{equation}
which implies that $\lambda_1(\cO, \sL+c)\geq \lambda_\cO$. However,
the equivalence of $\lambda_1(\cO, \sL+c)$ and $\lambda_\cO$ was
not settled in the literature so far. We shall first address this
issue here.
\begin{theorem}\label{AB}
For every smooth bounded domain $\cO$ we have $\lambda_1(\cO, \sL+c)=\lambda_{\cO}(\sL+c)$.
\end{theorem}
 The main difference in the definitions (\ref{E1.1}) and (\ref{5.13.1}) is that the former requires the test functions having positive infimum while  the later  only requires the test functions to be positive in $\mathcal{O}$, namely, it can be zero on $\partial\mathcal{O}$ (Dirichlet boundary condition). It is worth underlining that the set of admissible test functions in (\ref{E1.1}) includes the eigenfunctions with the Neumann boundary condition, whose existence was proved by Patrizi  \cite{SP}, therefore, from (\ref{E1.1}) $\lambda_{\cO}(\sL+c)$ must be bigger or equal to the Neumann eigenvalue obtained in \cite{SP}. 
Theorem~\ref{AB} shows  $\lambda_{\cO}(\sL+c)$  to be identical with the Dirichlet principal eigenvalue, namely the smallest real number satisfying (\ref{Dirichlet}). By Theorem \ref{AB}, we further show  that  $\lambda_{\cO}(\sL+c)$ can be defined on larger set of admissible test functions, which are only necessarily continuous and positive in $\mathcal{O}$. Let us also mention
two more interesting works \cite{JL05,JLM} that also studied the eigenvalue problem
in bounded domains of an equation involving
infinity Laplacian. In particular,
 Juutinen,  Lindqvist and Manfredi in \cite{JLM} first showed that the limiting behaviour of the 
ground state of $p$-Laplacian operators, as $p\to\infty$, and established its connection with certain
eigen-equation involving infinity Laplacian whereas \cite{JL05} considered the limit of the higher eigenvalues
of $p$-Laplacian. The simplicity of {\it first $\infty$-eigenvalue} are discussed in \cite{DRS,Y07}, but, up to our knowledge,  it has still been unknown. 

In this article, by proving a new Harnack inequality, we explore the notions of generalized principal eigenvalues, their limiting properties with respect to the parameters  and 
connection to the maximum principle of $\sL+c$ in $\RN$. Later, we  use this notions to investigate the existence, nonexistence and uniqueness of positive solutions of Fisher-KPP (for Kolmogorov, Petrovsky and Piskunov) type equations in the whole space. 

\begin{definition}
We define the generalized principal eigenvalue of $\sL+c$ 
as follows
$$\lambda_1(\RN, \sL+c)=\sup\{\lambda\in\R\; :\; \exists\;\text{positive}\; \psi\in\cC(\RN)\;
\text{satisfying}\; \sL\psi+ (c(x)+\lambda)\psi^{3-\gamma}\leq 0\; \text{in}\; \RN\}.$$
\end{definition}
For simplicity of notation we denote $\lambda_1(\RN, \sL+c)$ by $\lambda_1$ or $\lambda_1(\sL+c)$.
One of the interesting questions would be to find an admissible function attaining the value 
$\lambda_1$. Toward this goal, let us also denote by $\sE$ the set of all eigenvalues with positive eigenfunction i.e., $\lambda\in\sE$ if and only if there exists a positive $\psi\in\cC(\RN)$ satisfying
$$\sL\psi + (c(x)+\lambda)\psi^{3-\gamma}\,=\, 0 \quad\text{in}\; \RN.$$
Then we prove the following
\begin{theorem}\label{T1.1}
It holds that $\sE=(-\infty, \lambda_1]$, provided $\lambda_1>-\infty$.
\end{theorem}
This type of result has recently been proved by Berestycki-Rossi \cite[Theorem~1.4]{BR15}  for linear, non-degenerate elliptic operators in general smooth unbounded domains
with Dirichlet boundary condition (see also \cite{SA83}). Furusho and Ogura \cite[Proposition~4.2]{FO81}
also obtained
an analogous result for linear, non-degenerate elliptic operators defined in an exterior domain and satisfying
a more general boundary condition.
Similar result
for $p$-Laplacian is obtained by Nguyen-Vo \cite[Theorem~1.10]{NV19} whereas Biswas-Roychowdhury \cite[Theorem~2.1]{BR20} establish it for convex, fully nonlinear, non-degenerate  elliptic operators. 
The key tool in proving Theorem~\ref{T1.1} is the Harnack inequality for the operator $\sL+c$.
As far as we know the Harnack inequality is known only for non-negative infinite harmonic functions, that is,
for the nonnegative functions $u$ satisfying $\infdel u=0$. See for instance, \cite{E93,LM95} for the
Harnack inequality of infinite
harmonic functions and \cite{TB01} for infinite super-harmonic functions. One of the main contributions of this
article is to establish the Harnack inequality for a general operator. In the same spirit we also prove a
boundary Harnack estimate (or Carleson estimate) which might be of independent interest to a large number of readers. These estimates
are proved in Section~\ref{S-Har} and Section~\ref{S-BHI}.

Also, by Theorem~\ref{T1.1}, the non-triviality of $\sE$
 is linked with the finiteness of $\lambda_1$. It is evident that for
bounded $c$ we have $\lambda_1\geq -\sup_{\RN} c$. We discuss this issue again in Section~\ref{S-finite}.

Our  motivation in the study of the several generalized principal eigenvalues  comes from characterization of
the maximum principle and the existence/nonexistence of positive solutions of Fisher-KPP type equation in $\R^N$. We emphasize that, as pointed out in \cite{BR15}, the notion (\ref{5.13.1}) (i.e. $\lambda_1$) seems unsuitable to characterize the maximum principle in unbounded domains, and therefore, searching new quantities playing this role is an important question.   Denote by $\cC^+_{0}(\RN)$ the set of all positive continuous functions vanishing at 
infinity. As to be shown below,  the following two 
notions of generalized principal eigenvalue will play the  central role in the
validity of (weak) maximum principle in  the whole space.

\begin{definition} We define 
\begin{align*}
\lambda^\prime_1(\sL+c) &=\inf\{\lambda\in\R\;:\;  \, \psi\in\cC^+_0(\RN)
\; \text{satisfying}\; \sL\psi + (c(x)+\lambda)\psi^{3-\gamma}\geq 0 \; \text{in}\; \RN\}\,,
\\
\lambda^{\prime\prime}_1(\sL+c)&=\sup\{\lambda\in\R\;:\;  \, \varphi\in\cC(\RN)
\; \text{satisfying}\; \sL\psi + (c(x)+\lambda)\psi^{3-\gamma}\leq 0 \; \text{in}\; \RN,\; 
\\
&\hspace{10em}\text{and} \; \inf_{\RN}\psi>0 \}.
\end{align*}
\end{definition}
Our next result is the following relation.
\begin{theorem}\label{T1.2} There holds $\pplam\leq\lambda^\prime_1$. If  $|H(x, q)|\leq C|q|^{3-\gamma}$, then
we also have $\plam\leq\lambda_1$.
\end{theorem}
The equivalence of the eigenvalues in Theorem~\ref{T1.2} should be observed. In the case of bounded domains, all three
notions of generalized principal eigenvalue coincide, as shown in  Theorem~\ref{AB}. But they might not be equal for unbounded domains, in particular, in $\RN$.
See also Example~\ref{Eg1.2} for a specific situation.

\begin{example}\label{Eg1.1} 
It is noted that there may not exist any admissible eigenfunction for  $\plam$. Let us suppose $\sL \psi=\infdel \psi$ and $c\equiv 0$. From
the Liouville property it is easy to see that $\lambda_1(\sL)=\pplam(\sL)=0$. 
Theorem~\ref{T1.2} implies $\plam(\sL)=0$. Since any bounded function satisfying $\infdel u=0$ in $\RN$
are constant, there is no function $u$ in $\cC^+_0(\RN)$ satisfying $\infdel u=0$ in $\RN$.
\end{example}

It is an interesting question to ask if these several types of eigenvalues coincide. The answer is negative, in general, and therefore we cannot say $\plam$ and $\pplam$ are also characterizations for 
$\lambda_1$. The following example will demonstrate our statement: 
\begin{example}\label{Eg1.2}
Suppose that $N=1$ and $\sL u = (u^\prime)^2u^{\prime\prime}+ (u^\prime)^3$. For $\psi=e^{-\frac{1}{2}x}$ it is easy to see that 
$$\sL \psi + \frac{1}{16}\psi^3=0\quad \text{in}\; \R.$$
Then from the definition we obtain $\lambda_1(\sL)\geq \frac{1}{16}$. Next we show that 
$$\pplam(\sL)\,=\,\plam(\sL)\,=\,0.$$
 It is easy to see that $\pplam(\sL)\geq 0$. Thus, in view of
Theorem~\ref{T1.2}, it is enough to show that $\plam(\sL)\leq 0$. Consider any $\delta>0$. Let 
$g(x)=\frac{1}{1+\alpha x^2}$ for $\alpha>0$. An easy calculation reveals
\begin{align*}
(g^\prime)^2g^{\prime\prime}+ (g^\prime)^3 + \delta g^3(x)
&= \frac{32 (\alpha x)^4}{(1+\alpha x^2)^7} - \frac{8\alpha^3x^2}{(1+\alpha x^2)^6}
-\frac{8\alpha^3x^3}{(1+\alpha x^2)^6} + \delta g^3(x)
\\
&\geq \underbrace{\left(\delta - \frac{8\alpha^3x^2}{(1+\alpha x^2)^3}-\frac{8\alpha^3|x|^3}{(1+\alpha x^2)^3}\right)}_{\df F(x)}g^3(x)\,.
\end{align*}
We can choose $\alpha$ small enough so that $F(x)>0$ for all $x$. To see this, note that for $|x|\leq 1$,
$F(x)\geq \delta - 16\alpha^3$ and for $|x|\geq 1$, 
$$F(x)\geq \delta - \frac{16\alpha^3|x|^3}{(1+\alpha x^2)^3}\geq \delta - 
\frac{16\alpha(1+\alpha |x|^2)^2}{(1+\alpha x^2)^3}\geq \delta - \frac{16\alpha}{(1+\alpha)}.$$
Thus, for small $\alpha>0$, we have 
$$(g^\prime)^2g^{\prime\prime}+ (g^\prime)^3 + \delta g^3(x)\geq 0,$$
which yields $\plam(\sL)\leq\delta$. The arbitrariness of $\delta$ implies that $\plam(\sL)\leq 0$. 
Hence we have $\pplam(\sL)=\plam(\sL)=0$.
\end{example}

It  still remains an open question that whether  $\plam(\sL+c)$ and $\pplam(\sL+c)$ are equivalent.
It is worth mentioning that in case of linear, non-degenerate elliptic operator the equality of 
$\plam(\sL+c)$ and $\pplam(\sL+c)$ is conjectured by Berestycki-Rossi \cite[Conjecture~1.8]{BR15} and they were able to prove this conjecture in some special cases, in particular, for self-adjoint, uniformly elliptic operators in dimension $1$ or with  radially symmetric coefficients, thanks to the advantage of variational structure. Here, for a very general degenerate operator, we are able to prove the equivalence in some cases as stated below :
\begin{theorem}\label{T1.3}
The following statements hold true :
\begin{itemize}
\item[(i)]Assume  $|H(x, q)|\leq b(x) |q|^{3-\gamma}$, $\lim_{|x|\to\infty} b(x)=0$ and
 $\liminf_{|x|\to\infty} c(x)>-\lambda_1$.   Then one has $\lambda_1=\pplam$.
\item[(ii)] Assume  $|H(x, q)|\leq C |q|^{3-\gamma}$
and $\limsup_{|x|\to\infty} c(x)\leq -\lambda_1$, then we have $\lambda_1=\plam=\pplam$. Moreover, if $\limsup_{|x|\to\infty} c(x)<-\lambda_1$, then $\lambda_1$ possesses a principal eigenfunction  in $\cC^+_0(\RN)$.
\item[(iii)] Assume  $|H(x, q)|\leq b(x) |q|^{3-\gamma}$, $\lim_{|x|\to\infty} b(x)=0$ and
 $\lim_{|x|\to\infty} c(x)$ exists. Then we have
$\lambda_1=\pplam$.
\end{itemize}

\end{theorem}
The decay of $b$ in Theorem~\ref{T1.3}(i) is important. Otherwise, we get a counterexample from
Example~\ref{Eg1.2}.


Next we state our weak maximum principle which is the extension of maximum principle in bounded domains
\cite[Definition~1.1]{BCPR}.
\begin{definition}[Maximum Principle]
We say that the operator $$F(D^2u, \grad u, u, x)= \sL u + c(x) (u_+)^{3-\gamma}$$
 satisfies a (weak) maximum principle (MP), if for any $u\in\cC(\RN)$ satisfying
$$F(D^2u, \grad u, u, x)\geq 0\quad \text{in}\; \RN, \quad \text{and}\; \limsup_{|x|\to\infty}u(x)\leq0,$$
we have $u\leq 0$ in $\RN$. 
\end{definition}
\begin{remark}\label{R1.1}
There are other possible ways to state the MP. We may also consider $u\in\cC(\RN)$ satisfying
\begin{equation}\label{ER1.1}
\sL u + c(x)\sgn (u) |u|^{3-\gamma}\geq 0\quad \text{in}\; \RN, \quad \text{and}\; 
\limsup_{|x|\to\infty}u\leq 0.
\end{equation}
Then it is easy to see that 
$$\sL u + c(x) (u_+)^{3-\gamma}\geq 0\quad \text{in}\; \RN.$$
Thus, if MP holds, it can be applied to \eqref{ER1.1}.
\end{remark}
Our next result provides a necessary and a sufficient condition for the validity of MP. More precisely,
we  prove 
\begin{theorem}\label{T1.4} 
We have the following criteria to characterize the maximum principle in $\R^N$ :
\begin{itemize}
\item[(i)] If $\sL+c$ satisfies MP then it must hold that $\plam\geq 0$. Also, $\sL+c$ satisfies a MP if $\pplam>0$.

\item[(ii)] Let $\eta\df\limsup_{|x|\to\infty} c(x)<0$ and $|H(x, q)|\leq C |q|^{3-\gamma}$.
Then $\sL+c$ satisfies MP if and only if $\lambda_1>0$.
\end{itemize}
\end{theorem}

As it can be seen from Theorem~\ref{T1.4} that the eigenvalues $\plam$ and $\pplam$ play the key role in the validity of the maximum
principle. Under some specific conditions (Theorem~\ref{T1.4}(ii)) we can show that $\lambda_1$ characterizes the MP and it should be compared with \cite[Proposition~1.11]{BR15}.
The sign of $\eta$ in the above result is crucial. For example, suppose that $N=1,\gamma=0$,
$H(x, q)= q^3$ and $c=\frac{1}{32}$. Then it follows from  Example~\ref{Eg1.2} that 
$\lambda_1(\sL+c)\geq \frac{1}{32}$ but $\plam(\sL+c)=-\frac{1}{32}$. Thus MP does not hold for 
$\sL+c$, by Theorem~\ref{T1.4}. Another interesting consequence of Theorem~\ref{T1.4} and Example~\ref{Eg1.2}
is that validity of MP does not imply equality of eigenvalues.

%
%
%

The  next result concerns the limiting properties of the principal eigenvalue. Let us define following operators for the sake of presentation :
$$\sL_\alpha u \df \sL u + \alpha c(x) u^{3-\gamma},\quad \text{and}\quad
\tilde\sL_{\alpha} u \df \alpha\ginfdel u + H(x, \grad u)+ c(x) u^{3-\gamma},$$
where $c$ is a bounded continuous function. Let $\lambda_{1,\alpha}, \tilde\lambda_{1,\alpha}$ be the principal eigenvalues associated to $\sL_\alpha$ and $\tilde\sL_{\alpha}$, respectively.

\begin{theorem}\label{T1.6}
We have the following limits of the principal eigenvalue with respect to the parameters :
\begin{itemize}
\item[(i)] $\lim_{\alpha\to\infty}\frac{\lambda_{1,\alpha}}{\alpha}=-\sup_{\RN} c$,
\item[(ii)] $\lim_{\alpha\to-\infty}\frac{\lambda_{1,\alpha}}{\alpha}=-\inf_{\RN} c$,
\item[(iii)] $\liminf_{\alpha\to\infty}\tilde\lambda_{1,\alpha}\geq
-\limsup_{|x|\to\infty} c(x)$ , provided $|H(x, q)|\leq C|q|^{3-\gamma}$.
\item[(iv)] $\limsup_{\alpha\to\infty}\tilde\lambda_{1,\alpha}\leq
-\liminf_{|x|\to\infty} c(x)$, provided $|H(x, q)|\leq g(x)|q|^{3-\gamma}$ for some
function $g$ with $\lim_{|x|\to\infty} g(x)=0$.
\end{itemize}
\end{theorem}
These limits are related to the Liouville type results. Indeed, in Theorem \ref{T1.8} below, we  prove that  (\ref{ET1.7A}) admits a unique positive, bounded  solution if $\lambda_1'(\sL+a)<0$ and admits
no positive bounded solution if  $\plam(\sL+a)>0$. Consider a natural function $c$ with $\sup_{\R^N}c>0$ and $\inf_{\R^N}c<0$ and suppose $\lambda_1'(\sL_\alpha+a)=\lambda_{1,\alpha}$ (under some additional conditions as in Theorem \ref{T1.3}). By Theorem \ref{T1.6}, (i)-(ii), one finds $\alpha^\star\gg1$ and $\alpha_\star\ll-1$ such that equation (\ref{ET1.7A}) admits a unique positive, bounded  solution if $\alpha>\alpha^\star$ and admits no positive bounded solution if $\alpha<\alpha_\star$. For the effect of large diffusion coefficient in Theorem \ref{T1.6} (iii)-(iv), the existence of positive solution of equation (\ref{ET1.7A}) depends strongly on the sign of $c(x)$ near infinity. More precisely, as $\alpha$ large enough, there does not exist any positive, bounded solution for equation (\ref{ET1.7A})  if $\limsup_{|x|\to\infty}c(x)<0$ while equation (\ref{ET1.7A})  always possesses a positive, bounded solution as $\liminf_{|x|\to\infty}c(x)>0$, and moreover the later  statement was proved for fixed $\alpha$ in our previous work \cite{BV20}.


Our final result concerns the existence/nonexistence and uniqueness of solutions for certain class of quasilinear equation in $\RN$. It is  interesting that the quantity $\plam$ plays the key role in this characterization. Before stating the result, let us mention some basic conditions for the nonlinearity $f$. Let $f:\RN\times[0, \infty)\to \R$ be a continuous function satisfying the 
followings :
\begin{itemize}
\item[(F1)] $s\mapsto f(x, s)$ is locally Lipschitz in $s$, locally uniformly in $x$.
\item [(F2)] $f(x, 0)=0$ and for some $M>0$ we have $f(x, M)\leq 0$.
\item[(F3)] For every $x\in\RN$, $\frac{f(x,s)}{s^{3-\gamma}}$ is strictly decreasing in $s$ and
$a(x)\df\lim_{s\searrow 0} \frac{f(x,s)}{s^{3-\gamma}}$ exists, uniformly in $x$. Furthermore, 
\begin{equation}\label{EF3}
\limsup_{|x|\to \infty}a(x)<0\,.
\end{equation}
\end{itemize}
A typical example of $f$ would be $s^{3-\gamma}b(x)(a(x)-s^p)$ for $p>0$ and $b$ is a bounded continuous
function satisfying $\inf_{\RN} b>0$. Our result reads as follows :
\begin{theorem}\label{T1.8}
Consider the following equation
\begin{equation}\label{ET1.7A}
\sL u + f(x, u)\,=\, 0\quad \text{for}\; x\in\RN.
\end{equation}
Suppose that $|H(x, q)|\leq C|q|^{3-\gamma}$. Then
\begin{itemize}
\item[(i)]  If $\plam(\sL+a)>0$,  equation \eqref{ET1.7A} does not admit any  nontrivial, non-negative, bounded solution.
\item[(ii)] If $\plam(\sL+a)<0$, equation \eqref{ET1.7A}  admits a unique positive, bounded solution.
\end{itemize}
\end{theorem}

In the literature, similar
nonlinearity $f$ has been considered widely for other operators: for instance, Berestycki et al. \cite{BHN, BHR} 
studied the problem with nondegenerate elliptic operators whereas Nguyen-Vo \cite{NV15}  used the variational approach to studied the $p$-Laplacian
version of the equation \eqref{ET1.7A}. Equation \eqref{ET1.7A} is treated earlier in \cite{BV20} for a different set of conditions on $f$. In particular, in \cite{BV20}, we   required $\liminf_{|x|\to\infty} a(x)>0$ and obtained existence and uniqueness of bounded solution by showing  that $\lim_{|x|\to\infty}u(x)>0$, which was crucial to invoke a comparison principle. The current framework is in contrast with \cite{BV20} since here we show that $\lim_{|x|\to\infty}u(x)=0$ (see Lemma~\ref{L3.1}) and therefore, the techniques used in \cite{BV20} cannot be applied to tackle the current problem.  The result of Theorem \ref{T1.8} is in fact a Liouville type result, in which the existence of positive solution is characterized by  the sign of a generalized principal eigenvalue $\plam(\sL+a)$ depending on the \textit{zero order term}. This is  theoretically a counterpart to the
work \cite[Theorem~5.1]{ALT}, where nonexistence of positive solution is proved with strong absorption nonlinearity and the work \cite{CP09}, where the existence and uniqueness of positive solution are proved with concave-convex nonlinearities.


Lastly, we emphasize that the main novelties in this work is that we deal with the problem with a very general operator involving
infinity Laplacian coupled with a nonlinear gradient term. This  is a complicated nonlinear, degenerate elliptic operator, which neither enjoys  variational
structure (exception of the case $\gamma = 2$ in the two dimensional space \cite{DGIMR}) nor the subadditive
structure as Pucci type nonlinear operators.
Therefore most
of the existing approaches fail to apply in our framework .  For instance, if
we consider two solutions $u, v$, then the inequality satisfied by $u-v$ is not easy to
obtain. Such estimates are crucial to establish uniqueness.
 Many new techniques must be figured out and used
in a clever fashion to establish our results.

\textit{The paper is organized as follows : In the next section we prove the Harnack inequality and Section~\ref{S-Proof} is devoted to the proofs of our main results. Boundary Harnack inequality is established in Section~\ref{S-BHI}.}

\section{Harnack inequality}\label{S-Har}
In this article, we deal with the viscosity solution to the equations of the form
\begin{equation}\label{E2.1}
\sL u + F(x, u)\,=\, 0\quad \text{in}\;  \mathcal{O}, \quad \text{and}
\quad u=g\quad \text{on}\; \partial\mathcal{O},
\end{equation}
where $F$ and $g$ are assumed  to be continuous. For a symmetric matrix $A$ we define
$$M(A)=\max_{\abs{x}=1} \langle x, A x\rangle, \quad m(A)= \min_{\abs{x}=1} \langle x, A x\rangle.$$
We denote by $\sB_r(z)$ the open ball of radius $r$ centered at $z$   and when $z=0$, we simply write this
ball by $\sB_r$.
We use the notation $u\prec_{z}\varphi$ when $\varphi$ touches $u$ from above exactly at the point $z$ i.e.,
for some open ball $\sB_r(z)$ around $z$ we have $u(y)<\varphi(y)$ for $y\in\sB_r(z)\setminus\{z\}$ and
$u(z)=\varphi(z)$.
\begin{definition}[Viscosity solution]
An upper-semicontinuous (lower-semicontinous) function $u$ on $\bar{\mathcal{O}}$ is said to be a viscosity sub-solution (super-solution) of \eqref{E2.1}  if the followings are satisfied :
\begin{itemize}
\item[(i)] $u\leq g$ on $\partial \mathcal{O}$ ($u\geq g$ on $\partial \mathcal{O}$);
\item[(ii)] if $u\prec_{x_0}\varphi$ ($\varphi\prec_{x_0} u$ ) for
some point $x_0\in\mathcal{O}$ and a $\cC^2$ test function $\varphi$, then 
\begin{align*}
& \sL \varphi(x_0) + F(x_0, u(x_0))\,\geq\, 0\,,
\quad \left(\sL \varphi(x_0) + F(x_0, u(x_0))\leq\, 0,\; resp., \right);
\end{align*}
\item[(iii)] for $\gamma=2$, if $u\prec_{x_0}\varphi$ ($\varphi\prec_{x_0} u$) and $\grad\varphi(x_0)=0$ then 
\begin{align*}
& M(D^2\varphi(x_0)) + H(x,  \grad \varphi(x_0)) + F(x_0, u(x_0))\,\geq\, 0\,,
\\
&\left(m(D^2\varphi(x_0)) + H(x,  \grad \varphi(x_0)) + F(x_0, u(x_0))\leq\, 0,\; resp., \right)\,.
\end{align*}
\end{itemize}
We call $u$ a viscosity solution if it is both sub and super solution to \eqref{E2.1}.
\end{definition}
As well known, one can replace the requirement of strict maximum (or minimum) above by non-strict maximum (or minimum). We would also require the notion of superjet and subjet from \cite{CIL}. A second order \textit{superjet} of $u$ at $x_0\in\mathcal{O}$ is defined as
$$J^{2, +}_\mathcal{O} u(x_0)=\{(\grad\varphi(x_0), D^2\varphi(x_0))\; :\; \varphi\; \text{is}\; \cC^2\; \text{and}\; u-\varphi\; \text{has a maximum at}\; x_0\}.$$
The closure of a superjet is given by
\begin{align*}
\bar{J}^{2, +}_\mathcal{O} u(x_0)&=\Bigl\{ (p, X)\in\RN\times\bS^{d\times d}\; :\; \exists \; (p_n, X_n)\in J^{2, +}_\mathcal{O} u(x_n)\; \text{such that}
\\
&\,\qquad  (x_n, u(x_n), p_n, X_n) \to (x_0, u(x_0), p, X)\Bigr\}.
\end{align*}
Similarly, we can also define closure of a subjet, denoted by $\bar{J}^{2, -}_\mathcal{O} u$. See \cite{CIL} for more details.

Next we prove a Harnack inequality.
\begin{theorem}[Harnack inequality]\label{T2.1}
Let $u\in\cC(\bar\sB_1)$ be a non-negative solution to
\begin{equation}\label{ET2.1A}
\ginfdel u - \theta |\grad u|^{3-\gamma} - \mu u^{3-\gamma}\,\leq\, 0 \quad \text{in}\; \sB_1\,,
\end{equation}
for some non-negative $\theta, \mu$.
There exists a constant $C$, dependent only on $\mu, \theta, \gamma, N$,
satisfying
\begin{equation}\label{ET2.1B}
 u(x)\,\leq\, C u(y)\quad \text{for all}\; x, y\in\sB_{\frac{1}{2}}\,.
\end{equation}
\end{theorem}

\begin{proof}
Due to strong maximum principle \cite[Theorem~2.2]{BV20} we may assume that $u>0$ in $\sB_1$. Otherwise,
$u$ would be identically $0$ and (\ref{ET2.1B}) is trivial. The main idea of the proof is to construct a suitable subsolution and then use
comparison principle \cite[Theorem~2.1]{BV20}. A similar idea is also used in \cite{TB01} for infinite
super-harmonic functions.
Let $r>0$ be chosen
later. Fix $\alpha\in (0, 1)$ and define $v(x)=r^\alpha-|x|^\alpha$. Then an easy calculation gives us
$$\ginfdel v - \theta |\grad v|^{3-\gamma} - \mu v^{3-\gamma}\geq \alpha^{3-\gamma} (1-\alpha) |x|^{\alpha(3-\gamma)-4+\gamma} - \theta \alpha^{3-\gamma}|x|^{(\alpha-1)(3-\gamma)} - \mu r^{(3-\gamma)\alpha}\quad \text{in}\; \sB_r\setminus\{0\}\,.$$
Note that $\alpha(3-\gamma)-4+\gamma<(\alpha-1)(3-\gamma)<0$ and therefore, 
choosing
$$r<r_\circ\df \frac{1-\alpha}{2}\frac{1}{\theta + (\nicefrac{\mu}{\alpha^{3-\gamma}})^{\nicefrac{1}{4-\gamma}}},$$
we get
\begin{align*}
\ginfdel v - \theta |\grad v|^{3-\gamma} - \mu v^{3-\gamma} 
&\geq \alpha^{3-\gamma} |x|^{\alpha(3-\gamma)-4+\gamma}[(1-\alpha)-\theta |x|] - \mu r^{(3-\gamma)\alpha}
\\
&\geq \alpha^{3-\gamma} r^{\alpha(3-\gamma)-4+\gamma}[(1-\alpha)-\theta r] - \mu r^{(3-\gamma)\alpha}
\\
&= \alpha^{3-\gamma} r^{\alpha(3-\gamma)-4+\gamma} \left[1-\alpha - \theta r - \frac{\mu}{\alpha^{3-\gamma}}
r^{4-\gamma}\right]
\\
&>0,
\end{align*}
for $x\in \sB_r\setminus\{0\}$. Fixing any $r\in (0, r_\circ\wedge1)$ we show that 
\begin{equation}\label{ET2.1C}
u(x)\leq Cu(y)\quad \text{for all}\; x, y\in \sB_{r/4}.
\end{equation}
\eqref{ET2.1B} then follows from above and a standard covering argument. Pick $x\in \sB_{r/4}$
and let $w(z)= u(x) \frac{v(z-x)}{r^\alpha}$. Then we have $w(x)=u(x)$ and for $\Omega=\sB_r(x)\setminus\{x\}$
\begin{equation}\label{ET2.1D}
\ginfdel w(z) - \theta |\grad w(z)|^{3-\gamma} - \mu w^{3-\gamma}(z)\,\geq\,\zeta(z)\quad \text{in}\; \Omega,
\end{equation}
for some function $\zeta$ which is positive and continuous in $\Omega$. Using comparison principle
in \cite[Theorem~2.1]{BV20} we can show that $u\geq w$ in $\sB_r(x)$. We detail the arguments here for convenience since similar arguments will be used in several places of this article.

Suppose that $M\df\sup_{\Omega} (w-u)>0$. Since $w\leq u$ in $\partial\Omega$, it is evident that the value $M$
is attained inside $\Omega$. Now define the coupling function
$$\Theta_\varepsilon(x, y)= w(x)-u(y) -\frac{1}{4\varepsilon}\abs{x-y}^4 \quad \text{for}\; x, y\in\bar{\Omega}.$$
Note that the maximum of $\Theta_\varepsilon$ (say, $M_\varepsilon$) is larger than $M$ for all $\varepsilon$.
Let $(x_\varepsilon, y_\varepsilon)\in\bar\Omega\times\bar\Omega$ be a point of maximum for 
$\Theta_\varepsilon$.
It is then standard to show that (cf. \cite[Lemma~3.1]{CIL})
\begin{align*}
\lim_{\varepsilon\to 0} M_\varepsilon=M, \quad 
\lim_{\varepsilon\to 0}\frac{1}{4\varepsilon}\abs{x_\varepsilon-y_\varepsilon}^4=0.
\end{align*}
This clearly implies that $w(x_\varepsilon)-u(y_\varepsilon)\searrow M$, as $\varepsilon\to 0$. Again, since the maximizer  cannot move towards the boundary
we can find a subset $\mathcal{O}_1\Subset\Omega$ such that 
$x_\varepsilon, y_\varepsilon\in\mathcal{O}_1$ for all $\varepsilon$ small.
Now $u, v$ are Lipschitz continuous in $\mathcal{O}_1$, by \cite[Lemma~2.1]{BV20}, and
therefore we can find a constant
$\kappa$ such that
$$\abs{u(z_1)-u(z_2)} + \abs{w(z_1)-w(z_2)}\leq \kappa\abs{z_1-z_2}\quad z_1, z_2\in\mathcal{O}_1.$$
Observing
$$w(x_\varepsilon)-u(x_\varepsilon)\leq w(x_\varepsilon)-u(y_\varepsilon) -\frac{1}{4\varepsilon}\abs{x_\varepsilon-y_\varepsilon}^4,$$
we obtain 
\begin{equation}\label{ET2.1E}
\abs{x_\varepsilon-y_\varepsilon}^3\,\leq\, 4\varepsilon \kappa.
\end{equation}
Denote by $\eta_\varepsilon=\frac{1}{\varepsilon}\abs{x_\varepsilon-y_\varepsilon}^2(x_\varepsilon-y_\varepsilon)$ and $\beta_\varepsilon(x, y)=\frac{1}{4\varepsilon}\abs{x-y}^4$.
It then follows from \cite[Theorem~3.2]{CIL} that for some $X, Y\in\bS^{d\times d}$ we have $(\eta_\varepsilon , X)\in\bar{J}^{2, +}_\mathcal{O} w(x_\varepsilon)$,
$(\eta_\varepsilon, Y)\in\bar{J}^{2, -}_\mathcal{O} u(y_\varepsilon)$ and
\begin{equation}\label{ET2.1F}
\begin{pmatrix}
X & 0\\
0 & -Y
\end{pmatrix}
\leq 
D^2\beta_\varepsilon(x_\varepsilon, y_\varepsilon) + \varepsilon [D^2\beta_\varepsilon(x_\varepsilon, y_\varepsilon)]^2.
\end{equation}
In particular, we get $X\leq Y$. Moreover, if $\eta_\varepsilon=0$, we have $x_\varepsilon= y_\varepsilon$. Then from \eqref{ET2.1F} it follows that
\begin{equation}\label{ET2.1G}
\begin{pmatrix}
X & 0\\
0 & -Y
\end{pmatrix}
\leq 
\begin{pmatrix}
0 & 0\\
0 & 0
\end{pmatrix}.
\end{equation}
In particular, \eqref{ET2.1G} implies that $X\leq 0\leq Y$ and therefore, $M(X)\leq 0\leq m(Y)$.
Applying the definition of superjet and subjet we now obtain for $\eta_\varepsilon\neq 0$
\begin{align}
\zeta(x_\varepsilon)&\leq \abs{\eta_\varepsilon}^{-\gamma} \langle \eta_\varepsilon X, \eta_\varepsilon \rangle 
-\theta|\eta_\varepsilon|^{3-\gamma}-\mu (w(x_\varepsilon))^{3-\gamma}\nonumber
\\
& \leq \abs{\eta_\varepsilon}^{-\gamma} \langle \eta_\varepsilon Y, \eta_\varepsilon \rangle 
-\theta|\eta_\varepsilon|^{3-\gamma}-\mu (w(x_\varepsilon))^{3-\gamma}\nonumber
\\
&\leq \mu u^{3-\gamma}(y_\varepsilon)- \mu (w(x_\varepsilon))^{3-\gamma}\,.\label{ET2.1H}
\end{align}
Now, letting $\varepsilon\to 0$, we can find a point $z\in\cO_1$ so that $x_\varepsilon, y_\varepsilon\to z$,
up to a subsequence. This yields 
$$0<\zeta(z)\leq \mu (u^{3-\gamma}(z)-w^{3-\gamma}(z))\leq 0\,,$$
since $M=w(z)-u(z)> 0$. This is clearly a contradiction, and therefore, we get $w\leq u$ in $\sB_r(x)$.
Similar argument also works when $\eta_\varepsilon=0$.
Note that if $y\in \sB_{r/4}$, then $y\in \sB_{r/2}(x)$ and thus
$$u(y) \geq w(y)=u(x) (1-r^{-\alpha}|y-x|^\alpha)\geq u(x) (1- 2^{-\alpha}).$$
This gives \eqref{ET2.1C}.
\end{proof}

Combining the proof of Theorem~\ref{T2.1} and the ideas from \cite{TB07} we could also establish a boundary Harnack inequality.
Since we do not use this estimate in this article, we defer the proof to Section~\ref{S-BHI} as of independent interest, which may be used for other works in the future.

\section{Proofs of main results}\label{S-Proof}
In this section we prove our main results. Let us start discussing the finiteness of $\lambda_1$.
\subsection{Finiteness of $\lambda_1$}\label{S-finite}
We begin with a situation where $\lambda_1=-\infty$. We should compare this result with \cite[Proposition~2.6]{BR15}.
\begin{proposition}
Suppose that for some $C$ we have $|H(x, q)|\leq C |q|^{3-\gamma}$ for all $x, q$. Furthermore, assume that
for some $\delta>0$ there exists $x_n\in\RN$ such that
$$\lim_{n\to\infty}\; \inf_{z\in\sB_\delta(x_n)} c(z)=\infty.$$
Then we have $\lambda_1(\sL+c)=-\infty$.
\end{proposition}

\begin{proof}
Without any loss of generality we assume that $\delta=1$. Recall that 
$\lambda_{\mathcal{O}}(\sL+c)$ in \eqref{E1.1} decreases
with increasing domain. Thus it is enough to show that for any $m$ there exists a ball $\sB_1(x_n)$ such that
$\lambda_{\sB_1(x_n)}(\sL+c)\leq -m$.
Let $\varphi(x)=(e^{-k\abs{x}^2}-e^{-k})$ and $\varphi_n(x)=\varphi(x-x_n)$.
We show that for some large $k$  we have 
\begin{equation}\label{EP3.1A}
\ginfdel \varphi_n + H(x, \grad\varphi_n) + (c(x)-m)\varphi_n^{3-\gamma} \,>\, 0\quad \text{in}\; 
\sB_1(x_n).
\end{equation}
Then, by \eqref{E1.2}, it follows that $\lambda_{\sB_1(x_n)}(\sL+c) \leq -m$, which proves the result.
A direct computation yields, for $x\in\sB_1(x_n)$,
\begin{align*}
&\ginfdel \varphi_n + H(x, \grad\varphi_n) + (c(x)-m)\varphi_n^{3-\gamma}
\\
&\geq (2k\abs{x-x_n})^{4-\gamma}  e^{-(3-\gamma)k\abs{x}^2} - (2 k)^{3-\gamma} \abs{x-x_n}^{2-\gamma} e^{-(3-\gamma)k\abs{x}^2} 
\\
&\quad -C (2k\abs{x-x_n})^{3-\gamma}  e^{-(3-\gamma)k\abs{x-x_n}^2} + (c(x)-m) \left( e^{-k\abs{x-x_n}^2}-e^{-k}\right)^{3-\gamma}
\\
&\geq e^{-(3-\gamma)k\abs{x-x_n}^2} \Bigl[(2k\abs{x-x_n})^{4-\gamma} - (2 k)^{3-\gamma} \abs{x-x_n}^{2-\gamma} -
C (2k\abs{x-x_n})^{3-\gamma} 
\\
&\qquad + (\inf_{\sB_1(x_n)}c-m) \left(1-e^{-k(1-\abs{x-x_n}^2)}\right)^{3-\gamma}\Bigr].
\end{align*}
Now choose $k$ large enough so that for $1/2\leq \abs{x-x_n}\leq 1$ we have
$$(2k\abs{x-x_n})^{4-\gamma} - (2 k)^{3-\gamma} \abs{x-x_n}^{2-\gamma}
 -C (2k\abs{x-x_n})^{3-\gamma}\,>\, 0\,.$$
With this choice of 
$k$, we choose $n$ large so that for $|x-x_n|\leq 1/2$ we get 
\begin{align*}
&(2k\abs{x-x_n})^{4-\gamma} - (2 k)^{3-\gamma} \abs{x-x_n}^{2-\gamma} -
K (2k\abs{x-x_n})^{3-\gamma}
\\
&\qquad + \left[\inf_{\sB_1(x_n)}c-m\right] \left(1-e^{-k(1-\abs{x-x_n}^2)}\right)^{3-\gamma}>0.
\end{align*}
Note that for $\gamma=2$ we have to modify the calculation at $x=0$, but the estimate holds. 
Combining we have \eqref{EP3.1A}. This completes the proof.
\end{proof}

Next we provide sufficient conditions for finiteness of $\lambda_1$.
\begin{proposition}\label{P3.2}
Suppose that $H$ is anti-symmetric in $q$ (that is, $H(x, -q)=-H(x,q)$ for all $x,q$)
 and one of the following holds:
\begin{itemize}
\item[(a)] $H(x, x)\geq 0$ for large $|x|$ and 
\begin{equation}\label{EP3.2A}
\limsup_{|x|\to\infty}\frac{c(x)}{|x|^{\gamma-3}H(x, x)+1}\,<\, \infty\,.
\end{equation}
\item[(b)] $H(x, -x)\geq 0$ for large $|x|$ and 
\begin{equation*}
\limsup_{|x|\to\infty}\frac{c(x)}{|x|^{\gamma-3}H(x, -x)+1}\,<\, \infty\,.
\end{equation*}
\end{itemize}
Then $\lambda_1(\sL+c)>-\infty$.
\end{proposition}

\begin{proof}
Consider a $\cC^2$ convex function $\psi_1:\R\to[0, \infty)$ with the property that $\psi_1(r)=r^2$
for $|r|\ll 1$ and $\psi_1(r)=r$ for all $|r|\geq 1$. Let $\psi(x)=\psi_1(|x|)$. Consider (a) first. 
Using \eqref{EP3.2A} find $\sigma>0$ be such that
\begin{equation}\label{EP3.2B1}
\sigma^{3-\gamma} \left(|x|^{\gamma-3}H(x, x) + 1\right)\geq c(x) \quad \text{for all}\; |x|\geq 1.
\end{equation}
Let $g(x)=\exp(-\sigma\psi(x))$. An easy calculation shows for $x\neq 0$
\begin{align*}
\sL g+c(x) g^{3-\gamma} & = (-\sigma^{3-\gamma}\ginfdel\psi(x)  + \sigma^{4-\gamma}|\grad \psi|^{4-\gamma}
-\sigma^{3-\gamma} H(x, \grad\psi(x)) + c(x))g^{3-\gamma}(x).
\end{align*}
Since $\grad\psi(x)=\frac{x}{|x|}$ for all large $x$, it follows from \eqref{EP3.2B1} that
$$\sL g+c(x) g^{3-\gamma}\leq (-\sigma^{3-\gamma}\ginfdel\psi(x)  + \sigma^{4-\gamma}|\grad \psi|^{4-\gamma}
+(\sigma^{3-\gamma} + 1)\kappa)g^{3-\gamma}(x),$$
for some constant $\kappa$.
Since $|\grad \psi(x)|$ and $|\ginfdel\psi(x)|$ are bounded we can find a constant $\kappa_1$ such that
$\sL g+c(x) g^{3-\gamma}\leq \kappa_1 g^{3-\gamma}$ implying $\lambda_1(\sL+c)\geq -\kappa_1$. This completes 
the proof under (a). For (b), repeat the same argument with $g(x)=\exp(\sigma\psi(x))$.
\end{proof}

\subsection{Proofs of Theorems~\ref{AB}-\ref{T1.6}}
We continue with the proofs of our main results. We begin with the proof of Theorem~\ref{AB}.
\begin{proof}[Proof of Theorem~\ref{AB}]
By \cite[Theorem~3.1]{BV20} there exists a $\varphi\in \cC(\bar\cO)$ with $\varphi>0$ in $\cO$ and satisfying
\begin{equation}\label{EAB1}
\sL \varphi + (c(x)+\lambda_{\cO})\varphi^{3-\gamma}=0\quad \text{in}\; \cO, \quad
\varphi=0\quad \text{on}\; \partial\cO.
\end{equation}
Then from \eqref{5.13.1} it follows that $\lambda_1(\cO, \sL+c)\geq \lambda_\cO$. To complete the proof
we show that there is no $\lambda>\lambda_\cO$ and $\psi\in \cC(\bar\cO)$ with $\psi>0$ in $\cO$
satisfying
$$\sL \psi + (c(x)+\lambda)\psi^{3-\gamma}\leq 0\quad \text{in}\; \cO.$$
Suppose, on the contrary, that such $\psi$ exists for some $\lambda>\lambda_\cO$, that is,
\begin{equation}\label{EAB2}
\sL \psi + (c(x)+\lambda)\psi^{3-\gamma}\leq 0\quad \text{in}\; \cO, \quad \psi>0\quad \text{in}\; \cO.
\end{equation}
Normalize $\psi$ so that $\norm{\psi}_\infty=1$.
Let $\alpha\in (0, 1)$ be such that $\alpha^{3-\gamma}(c(x)+\lambda)>c(x)+\lambda_\cO$ for all $x\in\bar\cO$.
Define $v=\psi^\alpha$ and an easy calculation gives from \eqref{EAB2} that
$$\sL v +\frac{1-\alpha}{\alpha}\frac{1}{v}|\grad v|^{4-\gamma}+
 \alpha^{3-\gamma}(c(x)+\lambda) v^{3-\gamma} \leq\, 0\quad \text{in}\; \cO,$$
which implies
\begin{equation}\label{EAB3}
\sL v + \alpha^{3-\gamma}(c(x)+\lambda) v^{3-\gamma} \leq\, 0\quad \text{in}\; \cO\,.
\end{equation}
Using \cite[Theorem~2.2]{BV20} we obtain that $\psi(x)\geq \kappa \dist(x, \partial\cO)$ for $x\in\cO$,
for some constant $\kappa>0$. Hence
\begin{equation}\label{EAB4}
v(x)\;\geq\; \kappa^\alpha (\dist(x, \partial\cO))^\alpha\quad\text{for}\; x\in\cO.
\end{equation}

We next claim that $\varphi(x)\lesssim \dist(x, \partial\cO)$ if we normalize $\varphi$ to satisfy 
$\norm{\varphi}_\infty=1$. We write \eqref{EAB1} as
\begin{equation}\label{EAB5}
\sL\varphi \leq \norm{c+\lambda_\cO}_\infty\df C_1\quad \text{in}\;\cO.
\end{equation}
Since $|H(x, q)|\leq C |q|^{3-\gamma}$ for all $x\in\cO$, we see that for $\theta(x)=r^{-2}-|x|^{-2}$
we have
$$\ginfdel\theta(x) + C|\grad\theta|^{3-\gamma}
\leq - 3\, 2^{3-\gamma}|x|^{-10+3\gamma} + C 2^{3-\gamma} |x|^{-9+3\gamma}\quad \text{for}\; |x|>r\,.$$
Choosing $r$ small enough we see that
$$\sL \theta<C_1-1\quad \text{for}\; r<|x|\leq 2r.$$
We can translate $\theta$ to any point and base it on some annulus $D$ where the inner sphere touches 
$\cO$ from outside. Applying comparison principle \cite[Theorem~2.1]{BV20} in 
\eqref{EAB5} it follows that $\varphi\leq \kappa_1\theta$
in $D\cap\cO$, for some $\kappa_1$ depending on $r$. This implies that $\varphi(x)\leq \kappa_2 \dist(x, \partial\cO)$ for $x\in\cO$, for
some constant $\kappa_2$.

Define $w_1=\log v$ and $w_2=\log \varphi$. Since $\psi, \varphi$ are locally Lipschitz in $\cO$, we have
$w_1, w_2$ locally Lipschitz in $\cO$. Furthermore, we also have from \eqref{EAB1} and \eqref{EAB3} that
\begin{align}
\sL w_1 + |\grad w_1|^{4-\gamma} + \alpha^{3-\gamma}(c(x)+\lambda)& \leq \, 0,\label{EAB6}
\\
\sL w_2 + |\grad w_2|^{4-\gamma} + (c(x)+\lambda_\cO)&=0,\label{EAB7}
\end{align}
in $\cO$. Also, note from \eqref{EAB4} that
$$w_1(x)-w_2(x)=\log\frac{v(x)}{\varphi(x)}\geq \log 
\left[\frac{\kappa^\alpha}{\kappa_2}\dist(x, \partial\cO)^{\alpha-1}\right]\to \infty, $$
as $x$ approaches the boundary $\partial\cO$. Hence we can apply comparison principle (the proof of \cite[Theorem~2.1]{BV20}
works for these equations) to \eqref{EAB6}-\eqref{EAB7} to conclude that $w_1\geq w_2$ in $\cO$. But this is
not possible since $w_1-k$ also satisfies \eqref{EAB6} for any $k>0$. Therefore, there is no $\psi$ satisfying
\eqref{EAB2}. This completes the proof.
\end{proof}
Next we prove Theorem~\ref{T1.1}
\begin{proof}[Proof of Theorem~\ref{T1.1}]
We complete the proof in two steps.

{\bf Step 1.} Denote by $\lambda_n=\lambda_{\sB_n}(\sL+c)$
and let $\varphi_n$ be a positive eigenfunction corresponding to
it. Existence of such eigenfunction follows from \cite[Theorem~3.1]{BV20}. Then
\begin{equation}\label{ET1.1A}
\sL\varphi_n + c(x)\varphi_n^{3-\gamma}(x)=-\lambda_n\varphi^{3-\gamma}_n\quad \text{in}\; \sB_n,
\quad \text{and}\quad \varphi_n=0\; \text{on}\; \partial\sB_n.
\end{equation}
Normalize $\varphi_n$ by fixing $\varphi_n(0)=1$. Applying Harnack inequality, Theorem~\ref{T2.1}, we note
that for any compact $K$ we have a constant $C_K$ satisfying
$$\sup_{n}\; \sup_{K} \varphi_n\;\leq\; C_K.$$
Thus, by \cite[Lemma~2.1]{BV20}, we have $\{\varphi_n\}$ locally Lipschitz, uniformly in $n$. Hence
 we can extract a 
subsequence that converges to a locally Lipschitz function $\varphi\in \cC(\RN)$ and by stability property
of viscosity solution we would also have from \eqref{ET1.1A} that
$$\sL\varphi + c(x)\varphi^{3-\gamma}=-\tilde\lambda\varphi^{3-\gamma}\quad \text{in}\; \RN,$$
where $\tilde\lambda=\lim_{n\to\infty}\lambda_n$. Since $\lambda_1\leq\lambda_n$
for all $n$, $\tilde\lambda=-\infty$ implies that $\lambda_1=-\infty$.
Therefore, we must have $\tilde\lambda>-\infty$. It is also evident that $\tilde\lambda\geq\lambda_1$. Now by
strong maximum principle \cite[Theorem~2.2]{BV20} we get from above that $\varphi>0$ in $\RN$. Thus
$\varphi$ is an admissible function the definition of $\lambda_1$ implying 
$\lambda_1=\tilde\lambda$.

{\bf Step 2.} From Step 1 it follows that $\sE\subset (-\infty, \lambda_1]$. Pick $\lambda<\lambda_1$ and
consider the operator $\sL_\lambda=\sL+c+\lambda$. Then for every $n$, $\lambda_{\sB_n}(\sL_\lambda)=\lambda_n-\lambda>0$. We choose a sequence $\{f_n\}_{n\geq 1}$ of continuous, non positive, non-zero functions
satisfying 
\begin{align*}
\supp(f_n)\subset \sB_n\setminus \overline{\sB}_{n-1}\quad \text{ for all }\; n\in\mathbb{N}.
\end{align*}
Using the argument of \cite[Lemma~3.1]{BV20} we can then find a positive $u_n\in\cC(\bar\sB_n)$ satisfying
\begin{equation}\label{ET1.1B}
\sL u_n + c(x)u_n^{3-\gamma}(x)=-\lambda u_n^{3-\gamma}+ f_n(x)\quad \text{in}\; \sB_n,
\quad \text{and}\quad u_n=0\; \text{on}\; \partial\sB_n.
\end{equation}
Again, normalize $u_n$ to $u_n(0)=1$ and apply Harnack inequality to \eqref{ET1.1B}. Then repeat the argument
of Step 1 to find a positive solution $u$ to 
$$\sL u_n + c(x)u^{3-\gamma}(x)=-\lambda u^{3-\gamma}\quad \text{in}\; \RN.$$
Thus $\lambda\in\sE$, completing the proof.
\end{proof}

Next we prove Theorem~\ref{T1.4}.
\begin{proof}[Proof of Theorem~\ref{T1.4}(i)]
For the necessity part we see that if $\plam< 0$, there exists $\lambda<0$ and $\psi\in\cC^+_0(\RN)$
satisfying
$$\sL\psi + c(x)\psi^{3-\gamma}\geq -\lambda\psi^{3-\gamma}\geq 0\quad \text{in}\; \RN.$$
This clearly violates the MP. Hence validity of MP implies $\plam\geq 0$.

Now we prove the sufficiency part. Suppose that $\pplam>0$. Then we can find $\lambda>0$ and $V\in\cC(\RN)$
with $\inf_{\RN}V>0$ satisfying
\begin{equation}\label{ET1.3A}
\sL V + c(x) V^{3-\gamma} \leq -\lambda V^{3-\gamma}\quad \text{in}\; \RN.
\end{equation}
Consider $u\in\cC(\RN)$ satisfying
$$\sL u + c(x) (u_+)^{3-\gamma}\geq 0\quad \text{in}\; \RN, \quad \text{and}\; \limsup_{|x|\to\infty}u\leq0.$$
We show that $u\leq 0$ in $\RN$. On the contrary, we assume that $u$ is positive somewhere in $\RN$.
Let $\kappa=\sup_{\RN} \frac{u}{V}>0$. Since $\limsup_{|x|\to\infty}\frac{u}{V}\leq 0$, there exists a 
point $x_0\in\RN$ such that $u(x_0)=\kappa V(x_0)$. Replacing $V$ be $\kappa V$ we may assume that $\kappa=1$.

Now consider the coupling function
$$w_\varepsilon(x, y)= u(x)-V(y) -\frac{1}{4\varepsilon}|x-y|^4\quad x, y\in\RN.$$
Let $M_\varepsilon=\sup_{\RN\times\RN} w_\varepsilon\geq w_\varepsilon(x_0,x_0)=0$.
Since $u_+$ vanishes at infinity and
$V$ is bounded from below, there exists a compact set 
$K$ such that $w_\varepsilon(x, y)<0$ for $(x, y)\in (K\times K)^c$ and $\varepsilon\in (0,1)$. Thus we have 
$(x_\varepsilon, y_\varepsilon)\in K\times K$ satisfying 
$w_\varepsilon(x_\varepsilon, y_\varepsilon)=M_\varepsilon\geq 0$. By \cite[Lemma~2.1]{BV20}, both
$u, V$ are locally Lipschitz. Then from the arguments of Theorem~\ref{T2.1} we obtain
$$\lim_{\varepsilon\to 0} M_\varepsilon=0, \quad |x_\varepsilon-y_\varepsilon|^3\leq C_K\varepsilon,$$
and $x_\varepsilon, y_\varepsilon\to z\in K$, up to a subsequence, as $\varepsilon\to 0$. Furthermore,
$u(z)=V(z)>0$.
Repeating a calculation similar to \eqref{ET2.1H} we then have
$$-c(x_\varepsilon) u^{3-\gamma}(x_\varepsilon)-H(x_\varepsilon, \eta_\varepsilon)\leq -\lambda V^{3-\gamma}(y_\varepsilon)  -c(y_\varepsilon)V^{3-\gamma}(y_\varepsilon) - H(y_\varepsilon, \eta_\varepsilon),$$
where $\eta_\varepsilon=\frac{1}{\varepsilon}\abs{x_\varepsilon-y_\varepsilon}^2(x_\varepsilon-y_\varepsilon)$. Using \eqref{EH} this gives
us
$$0\leq -\lambda V^{3-\gamma}(y_\varepsilon) + c(x_\varepsilon) u^{3-\gamma}(x_\varepsilon)-c(y_\varepsilon)V^{3-\gamma}(y_\varepsilon)
+\omega_K(|x_\varepsilon-y_\varepsilon|) (1+|\eta_\varepsilon|^{3-\gamma}).$$
Letting $\varepsilon\to 0$
it follows that $0\leq -\lambda V^{3-\gamma}(z)$ contradicting the fact $\lambda>0$.
Therefore we must have $u\leq 0$ in $\RN$, completing the proof.
\end{proof}

Using Theorem~\ref{T1.4}(i) we can now prove Theorem~\ref{T1.2}. Before that let us define another 
intermediate quantity. By $\tilde\cC^+_0(\RN)$ let us denote the class of all non-negative, non-zero
continuous function vanishing at infinity. Define
$$\tilde\lambda^\prime_1(\sL+c) =\inf\{\lambda\in\R\;:\;  \, \psi\in\tilde\cC^+_0(\RN)
\; \text{satisfying}\; \sL\psi + (c(x)+\lambda)\psi^{3-\gamma}\geq 0 \; \text{in}\; \RN\}\,.$$
It is obvious that $\tilde\lambda^\prime_1(\sL+c)\leq \plam(\sL+c)$. We have the following
result.
\begin{theorem}\label{T3.1}
It holds that $\pplam\leq \tilde\lambda^\prime_1\leq\lambda_1$.
\end{theorem}

\begin{proof}
First we show $\tilde\lambda^\prime_1\leq\lambda_1$.
Recall from Theorem~\ref{T1.1} the
eigenpair $(\varphi_n, \lambda_n)$ satisfying
\begin{equation}\label{ET3.1A}
\sL\varphi_n + c(x)\varphi_n^{3-\gamma}(x)=-\lambda_n\varphi^{3-\gamma}_n\quad \text{in}\; \sB_n,
\; \varphi_n>0\;\; \text{in}\; \sB_n,
\quad \text{and}\quad \varphi_n=0\; \text{on}\; \partial\sB_n.
\end{equation}
Extend $\varphi_n$ in $\RN$ by $0$ i.e., let 
\[
\psi_n(x)=\left\{
\begin{array}{ll}
\varphi_n(x) & \text{for}\; x\in\sB_n\,,
\\
0 & \text{otherwise}.
\end{array}
\right.
\]
Note that $\psi_n\in\tilde\cC^+_0(\RN)$. We show that 
\begin{equation}\label{ET3.1B}
L\psi_n + c(x)\psi_n^{3-\gamma}(x)+\lambda_n\psi^{3-\gamma}_n\geq 0 \quad \text{in}\; \RN\,.
\end{equation}
Once \eqref{ET3.1B} is established, it follows from the definition of $\tilde\lambda^\prime_1$ that $\tilde\lambda^\prime_1\leq\lambda_n$
for all $n$. Therefore, the proof follows by letting $n\to\infty$ and using the fact that 
$\lim_{n\to\infty}\lambda_n=\lambda_1$. 
Thus we just have to show \eqref{ET3.1B}.
Let $\psi_n\prec_{x_0}\varphi$ for some $\varphi\in\cC^2(\RN)$ and $x_0\in\RN$. If $\varphi(x_0)=\psi_n(x_0)=0$,
then $\psi_n$ being non-negative, it follows that $\varphi$ has a local minimum at $x_0$. Thus
$D^2\varphi(x_0)\geq 0$ and $\grad\varphi(x_0)=0$. This also implies $M(D^2\varphi(x_0))\geq 0$.
Hence 
$$\sL\varphi(x_0) + c(x_0)\psi^{3-\gamma}_n(x_0)+\lambda_n\psi^{3-\gamma}_n(x_0)\geq 0.$$
On the other hand, if $\varphi(x_0)>0$ then $x_0\in \sB_n$ and the claim follows from \eqref{ET3.1A}.

Next we show $\pplam\leq\tilde\lambda^\prime_1$. Suppose, on the contrary, that 
$\pplam>\tilde\lambda^\prime_1$. 
Pick $\lambda\in (\tilde\lambda^\prime_1, \pplam)$ and denote by
$F=\sL+c+\lambda$. Then $\pplam(F)=\pplam-\lambda>0$. Since $\lambda>\tilde\lambda^\prime_1$, there exists $u\in\tilde\cC^+_0(\RN)$
satisfying
$$F(D^2 u, \grad u, u, x)\geq 0\quad \text{in}\; \RN.$$
Therefore, by Theorem~\ref{T1.4}(i), we must have $u\leq 0$ which is a contradiction. Hence we must have $\pplam\leq\tilde\lambda^\prime_1$.
\end{proof}
Now we prove Theorem~\ref{T1.2}
\begin{proof}[Proof of Theorem~\ref{T1.2}]
In view of Theorem~\ref{T3.1}, we only need to show that $\plam\leq \lambda_1$.
Considering any $\lambda>\lambda_1$ and we show that $\plam\leq\lambda$. This  implies that
$\plam\leq \lambda_1$.

Replacing $c$ be $c+\lambda$ we may assume that $\lambda=0$ and $\lambda_1<0$.
The method in \cite[Theorem~1.7(ii)]{BR15}
gives a bounded subsolution which might not vanish at infinity. So we need to modify the argument
suitably. Let $\{\xi_k\}$ be an
increasing sequence of non-negative cut-off functions converging to $1$, uniformly on compacts.
Define $c_k(x)= c_+(x)\xi_k(x)-c_-(x)$ and $\beta_k=\lambda_1(\sL+c_k)$. It is easily seen that
$\beta_k$ is a decreasing sequence, bounded below by $\lambda_1(\sL+c)$. We claim that 
\begin{equation}\label{ET1.2A}
\lim_{k\to\infty}\beta_k=\lambda_1(\sL+c).
\end{equation}
Let $\zeta_k$ be a principal eigenfunction of $\sL+c_k$ i.e.,
$$\sL\zeta_k + c_k(x)\zeta^{3-\gamma}_k+\beta_k\zeta^{3-\gamma}_k=0\quad \text{in}\; \RN.$$
Normalize $\zeta_k(0)=1$. Then using Harnack inequality (Theorem~\ref{T2.1}) and local Lipschitz
estimate (see Step 1 of Theorem~\ref{T1.1})
we can extract a subsequence of $\zeta_k$, converging to some positive $\varphi$, 
uniformly on compacts. Passing to the limit we obtain
$$\sL\varphi + c(x)\varphi^{3-\gamma}+(\lim_{k\to\infty}\beta_k)\varphi^{3-\gamma}=0\quad \text{in}\; \RN,$$
implying $\lambda_1\geq \lim_{k\to\infty}\beta_k$. This gives us \eqref{ET1.2A}.

From \eqref{ET1.2A} choose $k_0$ large enough so that $\beta_k<0$ for all $k\geq k_0$.
Since $\lim_{n\to\infty}\lambda_{\sB_n}(\sL+c_k)=\beta_k$, there exists
large $n$, say $n_0$, so that $\lambda_{\sB_{n_0}}(\sL+c_k)=\lambda_{n_0,k}<0$. Let
$\psi_{n_0}$ be such that
\begin{equation}\label{ET1.2B}
\sL\psi_{n_0} + c(x)\psi_{n_0}^{3-\gamma}(x)=-\lambda_{n_0}\psi^{3-\gamma}_{n_0}\quad \text{in}\; \sB_{n_0},
\psi_{n_0}>0\; \text{in}\; \sB_{n_0},
\quad \text{and}\quad \psi_{n_0}=0\; \text{on}\; \partial\sB_{n_0}.
\end{equation}
Denote by $\delta=\sup_{\sB_{n_0}} |c_k(x)|+1$. Normalize $\psi_{n_0}$ so that
$$
\norm{\psi_{n_0}}_{L^\infty(\sB_{n_0})}=\min\left\{1, \frac{-\lambda_{n_0}}{\delta}\right\}\,.
$$ 
Then it is easy to check from \eqref{ET1.2B} that $\psi_{n_0}$ is a subsolution to
\begin{equation}\label{ET1.2C}
\sL u + c_ku^{3-\gamma} = ((c_k(x))_+ +1)u^{4-\gamma}\quad \text{in}\; \sB_{n_0}\,.
\end{equation}
Also, $1$ is a super-solution in $\RN$. This subsolution and supersolution is enough to
construct a solution $u$, $0<u\leq 1$, to the equation
\begin{equation}\label{ET1.2CA}
\sL u + c_ku^{3-\gamma} = ((c_k(x))_+ +1)u^{4-\gamma}\quad \text{in}\; \RN\,.
\end{equation}
This can be obtained by constructing a positive solution $u_m$ in
each ball $\sB_m\supset \sB_{n_0}$, and then using an interior
Lipschitz estimate, we can show that
$u_m\to u$ along some subsequence, as $m\to\infty$, where
$u$ solves \eqref{ET1.2CA} (cf. the arguments in \cite[Lemma~4.1]{BV20}). 
Thus we have a positive bounded solution to \eqref{ET1.2CA}. Next we show that $\lim_{|x|\to\infty} u=0$.
Suppose, on the contrary, that $\limsup_{|x|\to\infty} u=2\kappa\in (0,1]$. By our construction, we observe that
$c_k(x)\leq 0$ outside a big ball $\sB$. Thus it follows from \eqref{ET1.2CA} that
\begin{equation}\label{ET1.2D}
\sL u \geq u^{4-\gamma} \quad \text{in}\; \sB^c.
\end{equation}
Let $r$ be a large number, so that $\zeta(x)= r^{-2}|x-x_0|^2-1$ satisfies
$$\sL\zeta \leq 8r^{(-6+2\gamma)}|x-x_0|^{2-\gamma} + 
C 2^{3-\gamma} r^{-2(3-\gamma)}|x-x_0|^{3-\gamma}<(\kappa/2)^{4-\gamma},$$
in $\sB_r(x_0)$, where we use that fact that $|H(x, q)|\leq C|q|^{3-\gamma}$.
 Fix $x_0$ large enough so that $u(x_0)>3\kappa/2$ and define 
$$\eta=\inf\{t\in [1, 3]\; :\; t+\zeta> u\quad \text{in}\; \sB_r(x_0) \}.$$
We can choose the point $x_0$ far enough to ensure that $\sup_{\sB_r(x_0)}u\leq 5\kappa/2$.
It is evident that $\eta\geq 1+3\kappa/2$, since $1+3\kappa/2+\zeta(x_0)=3\kappa/2<u(x_0)$. Also, 
$$\eta+\zeta\geq 1+3\kappa/2\geq 2\kappa+3\kappa/2= 7\kappa/2>\psi\quad\text{on}\; \partial\sB_r(x_0).$$
Define $v=\eta+\zeta$ and note that $\sL v\leq (\kappa/2)^{4-\gamma}$ in $\sB_r(x_0)$. Also,
$v$ touches $u$ from above a point $\tilde{x}\in\sB_r(x_0)$ and 
$u(\tilde{x})=v(\tilde{x})\geq 1+3\kappa/2+\zeta(\tilde{x})\geq 2\kappa$. 
Now consider a coupling function (as in Theorem~\ref{T2.1})
$$w_\varepsilon(x, y)= u(x)-v(y) -\frac{1}{4\varepsilon}\abs{x-y}^4 \quad \text{for}\;
 x, y\in\bar{\sB}_r(x_0).$$
Continue as in Theorem~\ref{T2.1} (see also the proof of Theorem~\ref{T1.4}) to obtain from \eqref{ET1.2D}
that
$$(2\kappa)^{4-\gamma}\leq u^{4-\gamma}(z)\leq (\kappa/2)^{4-\gamma},$$
which is a contradiction. Thus we must have $\lim_{|x|\to\infty} u(x)=0$, implying $u\in\cC^+_0(\RN)$.
From \eqref{ET1.2C} we also have $\sL u + c_k u^{3-\gamma}\geq 0$, and thus $\plam(\sL+c_k)\leq 0$.
Therefore, we obtain 
$$\lambda^\prime_1(\sL+c)\leq \plam(\sL+c_k)\leq 0.$$
This completes the proof.
\end{proof}

Next we prove Theorem~\ref{T1.3}.
\begin{proof}[Proof of Theorem~\ref{T1.3}]
Let $ c(x)+\lambda_1>2\delta>0$ for $x\in K^c_1$, for some compact set $K_1$. Let $\psi>0$ be a principal eigenfunction corresponding to $\lambda_1$. That is,
\begin{equation}\label{EP1.1A}
\sL\psi + (c(x)+\lambda_1)\psi^{3-\gamma}=0\quad \text{in}\; \RN.
\end{equation}
We show that $\inf_{\RN}\psi>0$. This gives $\pplam\geq\lambda_1$ and the proof follows from Theorem~\ref{T1.2}.

By \cite[Proposition~4.1]{BV20} there exists $R_1, R_2>0$ such that for any $|x|\geq R_1$ we have
$$\sL\phi^x + \delta (\phi^x)^{3-\gamma}>0 \; \text{in}\; \sB_{R_2}(x), \; \phi^x>0 \; \text{in}\; \sB_{R_2}(x_0), \; \text{and}\; \phi^x=0 \; \text{on}\; \partial\sB_{R_2}(x),$$
where $\phi^x(y)=\phi(y-x)$ for some smooth function $\phi$. $\phi$ attains its maximum at $0$ and we may also
assume that $\phi(0)=1$.
Now let $K=\bar{\sB}_{2R_1+2R_2}$ and 
$\kappa=\frac{1}{2}\min_{K}\psi$. Choosing $R_1$ large enough we may assume that $K_1\subset \sB_{R_1}$.
We show that $\inf_{\RN}\psi\geq \kappa$. Let $z\in K^c$ and 
$\gamma:[0, 1]\to\RN$ be the line joining $0$ to $z$.
Define
$$t_0=\sup\{t\; :\; \kappa\phi(\cdot-\gamma(t))<\psi(\cdot)\quad \text{in}\; \sB_{R_2}(\gamma(t))\}.$$
To complete the proof we must show that $t_0=1$. On the contrary, suppose that $t_0\in (0, 1)$. It is
evident that $|\gamma(t_0)|\geq R_1+R_2$. Then $\xi\df\kappa\phi(\cdot-\gamma(t_0))$ touches $\psi$ from 
below inside $\sB_{R_2}(\gamma(t_0))$ as $\xi$ vanishes on the boundary of $\sB_{R_2}(\gamma(t_0))$.
Again, by \eqref{EP1.1A} we have
$$\sL\psi + 2\delta \psi^{3-\gamma}\leq 0\quad \text{in}\; \sB_{R_2}(\gamma(t_0)).$$
Now repeating the argument of coupling in comparison principle (see for instance, the argument in Theorem~\ref{T1.2}) we get a contradiction.
Therefore $t_0=1$ and this completes
the proof of (i).

Next we consider (ii). We show that there exists a principal eigenfunction in $\cC^+_0(\RN)$ and this will
give us the claim. Since $\lambda_n=\lambda_{\sB_n}\to\lambda_1$ (see Step 1 of Theorem~\ref{T1.1})
choose $k$ large enough so that 
$c(x)+\lambda_n<0$ for all $n> k$ and $x\in\sB_k^c$. Recall that $\varphi_n$ is an principal eigenfunction
in $\sB_n$.  Fix $\varphi_n(0)=1$ and let $m_n=\sup_{\sB_k}\varphi_n$. By Harnack inequality,
Theorem~\ref{T2.1}, it then follows that $1\leq m_n\leq \kappa$, for some $\kappa>0$, for all $n> k$.
Furthermore, $m_n$ is a strict supersolution to $\sL u + (c+\lambda_n)u^{3-\gamma}=0$ in $\sB_n\setminus\bar\sB_k$.
Thus by \cite[Theorem~2.1]{BV20} we get $\varphi_n\leq m_n\leq \kappa$ for all $n>k$. Now passing to
the limit, as done in Theorem~\ref{T1.1}, we obtain a positive $\psi$ satisfying
\begin{equation}\label{EP3.2B}
\sL\psi + (c+\lambda_1)\psi^{3-\gamma}=0\quad \text{in}\; \RN.
\end{equation}
It is evident that $\psi\leq \kappa$. Next we show that $\lim_{|x|\to\infty}\psi(x)=0$ which would conclude the proof. On the contrary, suppose that $\limsup_{|x|\to\infty}\psi>0$. Normalizing $\psi$ we may assume that
$\limsup_{|x|\to\infty}\psi=2$. Choose compact $K$ so that $\sup_{K^c}(c(x)+\lambda_1)=-2\delta$. Thus
$$\sL\psi\geq 2\delta \psi^{3-\gamma}\quad \text{in}\; K^c.$$
Let $r$ be a large number, so that $\xi= r^{-2}|x-x_0|^2-1$ satisfies
$$\sL\xi \leq 8r^{(-6+2\gamma)}|x-x_0|^{2-\gamma} + C2^{3-\gamma} r^{-2(3-\gamma)}|x-x_0|^{3-\gamma}<\delta,$$
in $\sB_r(x_0)$. Fix $x_0$ large enough so that $\psi(x_0)>3/2$ and define 
$$\beta=\inf\{t\in [1, 3]\; :\; t+\xi>\psi\quad \text{in}\; \sB_r(x_0) \}.$$
We can choose the point $x_0$ far enough to ensure that $\sup_{\sB_r(x_0)}\psi\leq 2$.
It is evident that $\beta\geq 5/2$, since $5/2+\xi(x_0)=3/2<\psi(x_0)$. Also, $\beta+\xi\geq 5/2>\psi$
in $\partial\sB_r(x_0)$. Define $v=\beta+\xi$ and note that $Lv\leq\delta$ in $\sB_r(x_0)$. Also,
$v$ touches $\psi$ from above a point $z$ and $\psi(z)=v(z)\geq 5/2+\xi(z)\geq 3/2$.
Thus applying the definition of viscosity solution we get
$$2\delta ({3}/{2})^{3-\gamma}\leq 2\delta (v(z))^{3-\gamma}=2\delta (\psi(z))^{3-\gamma}\leq \sL v(z)=\sL \xi(z)<\delta,$$
which is a contradiction.
 Hence we have $\psi\in\cC^+_0(\RN)$,
completing the proof.

Now we show that $\lambda_1=\pplam$ under the assumption that $\limsup_{|x|\to\infty}c(x)\leq-\lambda_1$.
It is enough to show that for any $\delta>0$ there exists $u\in\cC(\RN)$ with $\inf_{\RN}u>0$ and
satisfying
\begin{equation}\label{EP1.1A1}
\sL u + (c(x)+\lambda-\delta)u\leq 0\quad \text{in}\; \RN.
\end{equation}
Let $\psi$ be a principal eigenfunction of $\sL+c$ in $\RN$. Normalize $\psi$
to satisfy $\psi(0)=1$. Let $K$ be a compact set satisfying $(c(x)+\lambda_1-\delta/2)< 0$ in $K^c$. By
Harnack inequality we have a constant $\kappa>0$ so that $\sup_K \psi <\kappa$.
Define $\psi_\varepsilon(x)=\psi(x)+\varepsilon$ for $\varepsilon\in (0,1)$.
From convexity we note that for any $a, b>0$ we have
$$ a^{3-\gamma}\geq b^{3-\gamma} + (3-\gamma) b^{2-\gamma} (a-b),$$
which gives us
\begin{equation}\label{EP1.1B}
\psi_\varepsilon^{3-\gamma}(x)\leq \psi^{3-\gamma}(x) + \varepsilon \varrho ,
\end{equation}
in $K$, where $\varrho= (3-\gamma) (\kappa+1)^{2-\gamma}$.
Using \eqref{EP1.1B} we then get
\begin{align*}
(c(x)+\lambda_1-\delta/2)\psi^{3-\gamma}_\varepsilon(x) &\leq (c(x)+\lambda_1-\delta/2)\psi^{3-\gamma}(x)
+ (c(x)+\lambda_1-\delta/2)_+ \varepsilon \varrho
\\
&\leq (c(x)+\lambda_1)\psi^{3-\gamma}(x) + [\max_{K} (c(x)+\lambda_1)_+] \frac{1}{\min_K\psi^{3-\gamma}} \varepsilon \varrho \psi^{3-\gamma}_\varepsilon(x)
\\
&=(c(x)+\lambda_1)\psi^{3-\gamma}(x) + \varepsilon \varrho_1 \psi^{3-\gamma}_\varepsilon(x),
\end{align*}
for some constant $\varrho_1$. Choose $\varepsilon$ small enough to satisfy $\varepsilon \varrho_1<\delta/2$.
Putting this in \eqref{EP3.2B} we have
$$\sL\psi_\varepsilon + (c(x)+\lambda_1-\delta)\psi^{3-\gamma}_\varepsilon\leq 0\quad
\text{in}\; \RN\,,$$
which gives us \eqref{EP1.1A1}. From the arbitrariness
of $\delta$ is follows that $\pplam=\lambda_1$.

(iii) follows by combining (i) and (ii).
\end{proof}

Using Theorem~\ref{T1.3}(ii) we now prove Theorem~\ref{T1.4}(ii)
\begin{proof}[Proof of Theorem~\ref{T1.4}(ii)]
The converse direction is an application of Theorem~\ref{T1.3}(ii). If we suppose that 
$\lambda_1\leq 0<-\eta$ then by Theorem~\ref{T1.3}(ii) there exists a principal eigenfunction 
$\varphi\in\cC^+_0(\RN)$ satisfying $\sL\varphi + (c+\lambda_1)\varphi^{3-\gamma}=0$. This implies
$$\sL\varphi + c(x) \varphi^{3-\gamma}\geq 0\quad \text{in}\; \RN.$$
and therefore, MP is violated. Hence validity of MP implies $\lambda_1>0$.

Now assume that $\lambda_1>0$. We show that MP holds for $\sL+c$. Let
$u\in\cC(\RN)$ be such that
$$\sL u + c(x) (u_+)^{3-\gamma}\geq 0\quad \text{in}\; \RN, \quad \text{and}\; \limsup_{|x|\to\infty}u\leq 0.$$
We want to show that $u\leq 0$. On the contrary, suppose that $u_+\neq 0$. Let $K$ be a compact set satisfying
$c(x)<0$ for $x\in K^c$. From the definition of viscosity solution it follows that $\max u_+$ is attained
only in $K$. Let $\psi>0$ be a principal eigenfunction i.e.
\begin{equation}\label{ET1.4A}
\sL\psi + (c(x)+\lambda_1)\psi^{3-\gamma}=0\quad \text{in}\; \RN\,.
\end{equation}
For $\varepsilon>0$ we define $\Omega_\varepsilon=\{x\in\RN\; :\; u(x)>\varepsilon\}$. Choosing $\varepsilon$
small we may assume that $\Omega_\varepsilon\neq \emptyset$. Since, for $a\geq b\geq 0$ we have
$(a-b)^{3-\gamma}\geq a^{3-\gamma}-(3-\gamma) a^{2-\gamma} b$ (due to convexity), we note that
\begin{equation}\label{ET1.4B}
\sL u + c(x) (u-\varepsilon)^{3-\gamma}\geq \sL u + c(x) u^{3-\gamma} - c_+(x)(3-\gamma) u^{2-\gamma}\varepsilon
\geq -\varepsilon c_+(x)(3-\gamma) \sup_K |u|^{2-\gamma},
\end{equation}
in $\Omega_\varepsilon$. We choose $\varepsilon$ small enough so that 
$$\varepsilon c_+(x)(3-\gamma) \sup_K |u|^{2-\gamma}< \lambda_1\inf_{K}\psi^{3-\gamma}\quad \forall\; x.$$
This is possible since $c_+$ vanishes outside $K$.
We claim that $u-\varepsilon\leq \psi$ in $\Omega_\varepsilon$ for all 
$\varepsilon$ small. If not, then we have $(u-\varepsilon-\psi)_+\neq 0$ in $\Omega_\varepsilon$.
Define $v=t(u-\varepsilon)$ where
$$t=\sup\{s\; :\; s (u-\varepsilon)<\psi\; \text{in}\; \Omega_\varepsilon\}.$$  
Since $\Omega_\varepsilon$ is bounded, we have $t>0$ and by assumption $t<1$. Of course, $v$ touches 
$\psi$ from below in $\Omega_\varepsilon$ and we may consider the connected component where this contact happens.
Also, we note from \eqref{ET1.4B} that 
$$\sL v + c(x) v^{3-\gamma}\geq - \varepsilon t^{3-\gamma} c_+(x)(3-\gamma) \sup_K |u|^{2-\gamma}
\geq -\varepsilon c_+(x)(3-\gamma) \sup_K |u|^{2-\gamma}\quad \text{in}\; \Omega_\varepsilon.$$
Then applying comparison principle between $v$ and $\psi$ (as done in Theorem~\ref{T2.2})
we would find a point $z\in\Omega_\varepsilon$
with $v(z)=\psi(z)$ and 
$$-\varepsilon c_+(z)(3-\gamma) \sup_K |u|^{2-\gamma} - c(z) v^{3-\gamma}(z)
\leq -(c(z)+\lambda_1)\psi^{3-\gamma}(z).$$
This is clearly not possible by our choice of $\varepsilon$. Thus $u-\varepsilon\leq \psi$ in $\Omega_\varepsilon$ for all  $\varepsilon$ small. Letting $\varepsilon\to 0$ it implies 
$u\leq\psi$ in $\RN$. This is also true for any $\kappa\psi$ for $\kappa>0$, and it can hold only if
$u\leq 0$ in $\RN$. Hence the result.
\end{proof}

\begin{remark}
From the proof of Theorem~\ref{T1.4}(ii) we make an observation that $\pplam>0$ is not a necessary condition 
for $\sL+c$ to satisfy MP. Consider the Example~\ref{Eg1.2} i.e., $\sL u = (u^\prime)^2 u^{\prime\prime}+
(u^\prime)^3$ and $c=0$. As shown in Example~\ref{Eg1.2}, we have $\lambda_1>0$ and $\plam=\pplam=0$.
Let us now show that $\sL$ satisfies MP. Suppose $\sL u\geq 0$ in $\RN$ and $\limsup_{|x|\to\infty}u\leq 0$.
Since $u-\varepsilon$ is also a subsolution, it is easily seen that $\sL (u-\varepsilon)_+\geq 0$. Then
the proof of Theorem~\ref{T1.4} gives us $(u-\varepsilon)_+\leq \kappa\psi$ for every $\kappa, \varepsilon>0$ where
$\psi$ is a principal eigenfunction. This is possible only if $u\leq 0$ in $\RN$. Thus $\sL$ satisfies MP.
\end{remark}

\subsection{Qualitative properties of the eigenvalues and eigenfunctions}
In this section we discuss some qualitative properties of the eigenvalues and eigenfunctions. More precisely, we prove Theorem~\ref{T1.5}, Theorem~\ref{T1.6} and Theorem~\ref{T1.7}.

In the spirit of \cite[Proposition~5.1]{BNV}, we find the continuity property of the principal eigenvalues with respect to the coefficients. 
\begin{theorem}\label{T1.5}
$\lambda_1$ ($\plam,\, \pplam$) are Lipschitz continuous with respect to the potential $c$.
$\lambda_1$ (and $\pplam$)
is locally Lipschitz in $H$ the following sense: let $H_1, H_2$ be two Hamiltonians satisfying
$$|H_1(x, q)-H_1(x, q)|\leq \delta |q|^{3-\gamma}\quad \text{for all}\; x, q,\quad \text{and}\; \delta<1,$$
and if $\lambda_1(\sL_i)$ is the generalized principal  eigenvalue of $\sL_i =\ginfdel + H_i + c$, $i=1,2$, then we have
$$|\lambda_1(\sL_1)-\lambda_1(\sL_2)|\leq (3c_0 + d + 3|\lambda_1(\sL_1)|+3|\lambda_1(\sL_2)|)\delta,$$
where $\sup_{\RN}c= c_0$ and $d=\frac{1}{4-\gamma}\left[\frac{4-\gamma}{3-\gamma}\right]^{-(3-\gamma)}$.  
\end{theorem}

\begin{proof}
Proof of the first part follows from the definitions. We consider the second part and 
only do it for $\lambda_1$.
Let $(\lambda, u)$, $u>0$, be such that
$$\ginfdel u + H_1(x,\grad u) + c(x) u^{3-\gamma} +\lambda u^{3-\gamma}\leq 0\quad \text{in}\; \RN.$$
Denote by $v=u^\alpha$ for some $\alpha\in (0, 1)$, to be chosen later. It is routine to show that
$$\ginfdel v +\frac{1-\alpha}{\alpha}\frac{1}{v}|\grad v|^{4-\gamma} +H_1(x,\grad v) + \alpha^{3-\gamma}c(x) v^{3-\gamma} +\alpha^{3-\gamma}\lambda v^{3-\gamma}\leq 0\quad \text{in}\; \RN,$$
in viscosity sense. Then
\begin{align*}
\sL_2 v + \lambda v^{3-\gamma}
&\leq (c_0+\lambda) (1-\alpha^{3-\gamma})v^{3-\gamma} + \delta |q|^{3-\gamma}-\frac{1-\alpha}{\alpha}\frac{1}{v}|\grad v|^{4-\gamma} 
\\
&\leq (c_0+\lambda) (1-\alpha^{3-\gamma})v^{3-\gamma} + d \delta^{4-\gamma} (\frac{\alpha}{1-\alpha})^{3-\gamma}
v^{3-\gamma},\quad [\text{Young's inequality}].
\end{align*}
By convexity, we have $(1-\alpha^{3-\gamma})\leq (3-\gamma)(1-\alpha)$ and choose 
$\alpha=1-\delta$ to obtain
$$\sL_2 v + \lambda v^{3-\gamma}\leq [3(c_0+|\lambda|)+d]\delta v^{3-\gamma}\quad \text{in}\; \RN.$$
Choosing a sequence of $\lambda$ converging to $\lambda_1(\sL_1)$ we obtain
$$\lambda_1(\sL_2)\geq \lambda_1(\sL_1)- [3(c_0+|\lambda_1(\sL_1)|)+d]\delta.$$
Hence the result follows.
\end{proof}

Now we come to Theorem~\ref{T1.6}
\begin{proof}[Theorem~\ref{T1.6}]
Since (ii) follows from (i) by replacing $c$ by $-c$, we only prove (i). From the definition it is
evident that $\lambda_{1,\alpha}\geq -\alpha \sup_{\RN}c$. Now pick a ball $\sB$ satisfying 
$c(x)>\sup_{\RN}c-\varepsilon$ for all $x\in\sB$. Let $(\lambda_\sB, \varphi_\sB)$ be a principal
eigenpair of $\sL$ in $\sB$ i.e.,
$$\sL\varphi_\sB+\lambda_\sB\varphi^{3-\gamma}_\sB=0\quad \text{in}\; \sB.$$
Then it follows that, in $\sB$,
\begin{align*}
\sL\varphi_\sB + \alpha c(x) \varphi^{3-\gamma}_\sB &= (-\lambda_\sB+\alpha\varepsilon)\varphi^{3-\gamma}_\sB
+ \alpha(\sup c-2\varepsilon)\varphi^{3-\gamma}_\sB
\\
&\geq \alpha(\sup c-2\varepsilon)\varphi^{3-\gamma}_\sB,
\end{align*}
for all $\alpha$ large. Using \eqref{E1.2} and monotonicity property of eigenvalues we get
$$\lambda_1(\sL+\alpha c)\leq \lambda_{\sB}(\sL+\alpha c)\leq -\sup_{\RN} c + 2\varepsilon,$$
for all $\alpha$ large, implying $\limsup_{\alpha\to\infty}\frac{\lambda_{1,\alpha}}{\alpha}\leq -\sup_{\RN} c
+ 2\varepsilon$. Since $\varepsilon$ is arbitrary this completes the proof.

(iii)\, The idea of the proof is to construct a supersoltuion for every $\alpha$ large. Fix $\varepsilon>0$
and let $R$ be such that $c(x)-\eta-\varepsilon\leq 0$ for all $x\in \sB_R^c$ where $\eta=\limsup_{|x|\to\infty}c(x)$. Let $\chi:[0, \infty)\to [0, \infty)$ be a $\cC^2$, strictly increasing,
convex function such that $\chi(0)=\chi^\prime(0)=0$ and $\chi(r)=r$ for $r\geq 1$. Let $\varphi(x)=\chi(|x-x_0|)$ where $|x_0|=R+1$. Then $\varphi$ is also a convex function that is radially decreasing about $x_0$.
Let $\upsilon(x)=(e^{\beta\varphi(x)}+ e^{-\beta\varphi(x)})^{-\delta}$ where $\delta=\delta(\alpha)$ and
$\beta=\beta(\alpha)$ satisfy the following
\begin{equation}\label{EP1.2A}
\delta=\alpha^{-\frac{1}{4-\gamma}}, \quad \beta=\alpha^{\frac{-1}{2(4-\gamma)(6-\gamma)}}\,.
\end{equation}
By the above choice we have $\alpha\delta^{4-\gamma}\beta^{4-\gamma}\to 0$ and 
$\alpha\delta^{3-\gamma}\beta^{6-\gamma}\to\infty$ as $\alpha\to\infty$. This fact will be useful in the
calculation below.
Denote by 
$$h(x)=\frac{e^{\beta\varphi(x)}-e^{-\beta\varphi(x)}}{e^{\beta\varphi(x)}+e^{-\beta\varphi(x)}}.$$
Now a straightforward calculation gives us
\begin{align*}
\tilde\sL_{\alpha}\upsilon - \eta\upsilon^{3-\gamma} &\leq \upsilon^{3-\gamma}\Bigl[
\underbrace{\alpha \delta^{3-\gamma}(\delta+1)\beta^{4-\gamma} h^{4-\gamma} |\grad\varphi|^{4-\gamma}}_{A(\alpha)}-
\underbrace{\alpha \delta^{3-\gamma}\beta^{4-\gamma} h^{2-\gamma} |\grad\varphi|^{4-\gamma}}_{B(\alpha)}
\\
&\; -\underbrace{\alpha\delta^{3-\alpha}\beta^{3-\gamma}h^{3-\gamma}\ginfdel\varphi}_{C(\alpha)}
+ \underbrace{C \delta^{3-\gamma}\beta^{3-\gamma}h^{3-\gamma}}_{D(\alpha)} + c(x)-\eta
\Bigr]
\end{align*}
Since $\varphi$ is convex we have $D(\alpha)\geq 0$. Note that $\sB_R\subset \sB_{2R+1}(x_0)\setminus\sB_1(x_0)$ and hence, for all $x\in \sB_R$ we have
$$h(x)=\frac{e^{2\beta\varphi(x)}-1}{e^{2\beta\varphi(x)}+1}\geq
\frac{2\beta\chi(|x-x_0|)}{e^{2\beta\chi(2R+1)}+1}\geq \frac{2\beta}{e^{2\beta\chi(2R+1)}+1}.$$
Furthermore, for large $\alpha$ (i.e. small $\beta$) we have 
$$\sup_{\sB_R}[(1+\delta)h^2(x)-1]<-\frac{1}{2}.$$
Combining we get that for $x\in\sB_R$,  $|\grad\varphi|=1$ and
$$A(\alpha)-B(\alpha)+D(\alpha)+c(x)\leq \kappa_R \alpha\delta^{3-\gamma}\beta^{6-\gamma}(-\frac{1}{2})
+ C\, (\delta\beta)^{3-\gamma} + c(x)+\eta\to -\infty, $$
as $\alpha\to\infty$, using \eqref{EP1.2A}.
 So we consider $x\in\sB_R^c$. Recall $c(x)-\eta\leq \varepsilon$ in $\sB_R^c$.
Since $ 0\leq h(x)\leq 1$ we have
$$A(\alpha)-B(\alpha)\leq \alpha \delta^{4-\gamma}\beta^{4-\gamma} h^{2-\gamma} |\grad\varphi|^{4-\gamma}
\to 0,$$
as $\alpha\to\infty$, uniformly in $x$. Same limit holds for $D(\alpha)$. Thus we have
$$A(\alpha)-B(\alpha)-C(\alpha) + D(\alpha)+c(x)-\eta<2\varepsilon,$$
in $\sB_R^c$, as $\alpha\to\infty$. Combining we thus have
$$\tilde\sL_{\alpha}\upsilon + (-\eta-2\varepsilon)\upsilon^{3-\gamma}\leq 0\quad \text{in}\; \RN.$$
This implies that $\liminf_{\alpha\to\infty}\tilde\lambda_{1,\alpha}\geq \eta-2\varepsilon$. The
result follows from the arbitrariness of $\varepsilon$.

(iv)\, Denote by $\hat\eta=\liminf_{|x|\to\infty}c(x)$. From \cite[Proposition~4.1]{BV20}, for
every $\varepsilon>0$, there exists $R_1, R_2$, dependent on $\alpha$, satisfying for 
any $|x|\geq R_1$ that
$$\ginfdel\psi^x + \frac{1}{\alpha}H(y, \grad\psi^x)\phi^x + \frac{\varepsilon}{\alpha} (\phi^x)^{3-\gamma}>0 \; \text{in}\; \sB_{R_2}(x), \; \phi^x>0 \; \text{in}\; \sB_{R_2}(x_0), \; \text{and}\; \phi^x=0 \; \text{on}\; \partial\sB_{R_2}(x),$$
where $\phi^x(y)=\phi(y-x)$ for some smooth function $\phi$. Now choose $|z|$ large enough so that
$c(x)-\eta+2\varepsilon>\varepsilon$ in $\sB_{R_2}(z)$. Then we get from above that
$$\ginfdel\phi^z + \frac{1}{\alpha}H(y, \grad\phi^z)\phi^z + \frac{c(x)-\eta+2\varepsilon}{\alpha} 
(\phi^z)^{3-\gamma}>0
\quad \text{in}\; \sB_{R_2}(z).$$
Using \eqref{E1.2} we then have
$$\tilde\lambda_{1,\alpha}\leq \lambda_{\sB_{R_2}(z)}(\tilde\sL_{\alpha})\leq -\eta+2\varepsilon.$$
The result follows by letting $\alpha$ to infinity and using arbitrariness of $\varepsilon$.
\end{proof}

Next, by using the Harnack inequality, we can also obtain a lower bound on the decay of the eigenfunction.
\begin{theorem}\label{T1.7}
Suppose that $\varphi$ is a positive solution to $\sL\varphi + c(x)\varphi^{3-\gamma}=0$ in $\RN$ with $|H(x, q)|\leq C q^{3-\gamma}$
and $\sup|c|\leq \mu$. Then for any $\alpha\in (0, 1)$, we have
$$\varphi(x)\geq C_1 \min_{\sB_1}\varphi\, \exp\left(-4|x|\kappa_\alpha
\Bigl[C + (\nicefrac{\mu}{\alpha^{3-\gamma}})^{\nicefrac{1}{4-\gamma}}\Bigr]\right)\quad \text{for all $|x|$ large},$$
for some constant $C_1$, where
$$\kappa_\alpha=\frac{\log (1-2^{-\alpha})}{\alpha-1}.$$
\end{theorem}

\begin{proof}
Let $r=\frac{1-\alpha}{2}\frac{1}{C + (\nicefrac{\mu}{\alpha^{3-\gamma}})^{\nicefrac{1}{4-\gamma}}}.$
Pick any point $x\in \RN$ and consider the function $v(y)=\frac{\varphi(x)}{r^\alpha}(r^\alpha-|y-x|^\alpha)$.
Then, as shown in Theorem~\ref{T2.1}, we have
$$\sL v + c(x) v>0\quad \text{in}\; \sB_r(x)\setminus\{x\}.$$
By comparison principle (see the proof of Theorem~\ref{T2.1}) we then obtain 
$$\varphi(y) \geq v(y) \quad \text{in}\; \sB_r(x).$$
In particular,
\begin{equation}\label{ET1.6A}
\varphi(x) (1-2^{-\alpha})\leq u(y) \quad \text{for all}\; y\in\bar\sB_{r/2}(x).
\end{equation}
For any $x$ large, define $\eta=-\frac{x}{|x|}$ and $\xi_i=x-\eta {i}\frac{r}{2}$. For 
$m=\lfloor\frac{2|x|}{r}\rfloor + 1$, we have $\xi_m\in\sB_r(0)$. Successively applying \eqref{ET1.6A} we
obtain
$$\varphi(x)\geq (1-2^{-\alpha})\varphi(\xi_1)\geq\ldots\geq (1-2^{-\alpha})^{m}\min_{\sB_r(0)}\varphi.$$
We note that 
$$ (1-2^{-\alpha})^{m}=e^{m\log (1-2^{-\alpha}) }= e^{\frac{2|x|}{r}\log(1-2^{-\alpha})}
=\exp\left(-4|x|\{C + (\nicefrac{\mu}{\alpha^{3-\gamma}})^{\nicefrac{1}{4-\gamma}}\}\frac{\log (1-2^{-\alpha})}{\alpha-1}\right).$$
Again, by Harnack inequality, we also have $\min_{\sB_r(0)}\varphi\geq C_1 \min_{\sB_1(0)}\varphi$.
Combining we get the result.

\end{proof}

\subsection{Proof of Theorem~\ref{T1.8}}
The goal of this section is to prove Theorem~\ref{T1.8}. Two main ingredients for this purpose
are the comparison theorem (Theorem~\ref{T3.2}) and decay property at infinity (Lemma~\ref{L3.1}).
\begin{theorem}\label{T3.2}
Assume (F3) and the fact that $f(x, \cdot)$ is locally Lipschitz, uniformly with respect to $x$ in compacts.
Suppose that $\sL u + f(x, u)\geq 0$ in $\RN$ and 
$\sL v + f(x, v)\leq 0$ in $\RN$ where $u, v$ are positive continuous functions. In addition, also
assume that $\lim_{\abs{x}\to\infty} u(x)=0$.
Then we have $u\leq v$ in $\RN$.
\end{theorem}

\begin{proof}
Using \eqref{EF3} find a compact ball $K$ such that $a(x)<0$ for all $x\in K^c$. We assume that $(u-v)_+\neq 0$ in $\RN$ and arrive at a contradiction.  We divide the proof 
in cases.
\begin{itemize}
\item[(A)] $u\leq v$ in $K$.
\item[(B)] $\max_{K} \frac{u}{v}\geq 1+2\delta$, for some $\delta>0$.
\end{itemize}

{\bf Situation (A).}\,
We show that in this case $u\leq v$ in $\RN$ which would be a
contradiction to the fact that $(u-v)_+\neq 0$. Since $(u-v)_+\neq 0$, for some point $x\in K^c$ we have $u(x)-v(x)=2\varepsilon>0$. Let 
$v_\varepsilon(x)=v(x)+\varepsilon$. Then, we have $\liminf_{|x|\to\infty} v_\varepsilon(x)\geq \varepsilon>0$
and $u(x)-v_\varepsilon(x)=\varepsilon>0$. Define
$$t=\max\{s\geq 0\; :\; s u< v_\varepsilon\; \text{in}\; K^c\}.$$
Since $\lim_{|x|\to\infty} u(x)=0$, we have $t>0$.
Also, under the above circumstance, $t<1$. Let $\bar{u}=tu$. Then, $\bar{u}<v_\varepsilon$ on $\partial K$. We denote
by $\cO=K^c$. Since $a<0$ in $\cO$ it follows that $f(x, s)\leq s^{3-\gamma}a(x)<0$ in $\cO\times(0, \infty)$.
Thus, by monotonicity,
\begin{equation}\label{ET3.2A0}
f(x, s+\varepsilon)< \frac{(s+\varepsilon)^{3-\gamma}}{s^{3-\gamma}}f(x, s)\leq f(x, s)
\quad \text{in}\; \cO\times(0, \infty).
\end{equation}
Therefore, we have $\sL v_\varepsilon + f(x, v_\varepsilon)\leq 0$ in $\cO$. On the other hand, we 
also have
$\sL\bar{u} + f(x, \bar{u})\geq \sL\bar{u} + t^{3-\gamma}f(x, u)\geq 0$ in $\RN$.
Now consider the coupling function between $v_\varepsilon$ and $\bar{u}$ i.e.,
$$w_\kappa(x, y)= \bar{u}(x)-v_\varepsilon(y) - \frac{1}{4\kappa}|x-y|^2\quad x, y\in \bar\cO,
\quad \text{and}\; \kappa\in (0, 1)\,.$$
Let $(x_\kappa, y_\kappa)\in\argmax_{\bar\cO\times\bar\cO}w_\kappa$. It is evident that
$w_\kappa(x_\kappa, y_\kappa)\geq 0=\max_{\bar\cO} (\bar{u}-v_\varepsilon)$.
 Let $R$ be a large number so that
$$w_\kappa(x, y)\leq \bar{u}(x)-\varepsilon<-\varepsilon/2\quad \text{for}\; |x|\geq R,\; y\in\bar\cO\,.$$
Therefore, enlarging $R$, if required, we get $w_\kappa(x, y)<-\varepsilon/4$ whenever $(x, y)\in(\sB_R\times\sB_R)^c$.  This readily gives us $(x_\kappa, y_\kappa)\in (\sB_R\times\sB_R)\cap (\bar\cO\times\bar\cO)$. Now use the arguments in Theorem~\ref{T2.1} to obtain that
\begin{equation}\label{ET3.2A}
\lim_{\kappa\to 0} w_\kappa(x_\kappa, y_\kappa)=0,\quad
|x_\kappa-y_\kappa|^3\leq L \kappa,
\end{equation}
for some constant $L$, and $x_\kappa, y_\kappa\to z\in \bar\sB_R\cap\cO$ as $\kappa\to 0$, along some
subsequence. Using \eqref{ET3.2A} and an argument similar to the proof of Theorem~\ref{T1.4} we arrive at
$$-t^{3-\gamma} f(z, u(z))\leq -f(z, v_\varepsilon(z)),$$
and $\bar{u}(z)=v(z)$. This gives us
$$f(z, v_\varepsilon(z))-t^{3-\gamma} f(z, \frac{v_\varepsilon(z)}{t^{3-\gamma}})\leq 0,$$
which contradicts the strictly decreasing property of $\frac{f(x, s)}{s^{3-\gamma}}$ in $s$. Thus we must have
$u\leq v$ in $\RN$, as claimed.

{\bf Situation (B).} We show that $\max_{K} \frac{u}{v}\geq 1+2\delta$ leads to a contradiction.
Fix $\varepsilon_0>0$ such that
\begin{equation}\label{ET3.2B}
\min_{\varepsilon\leq \varepsilon_0} \max_{K}\frac{u}{v_\varepsilon}\geq 1+\delta.
\end{equation}
Let 
$$t=t_\varepsilon\df\max\{s\; :\; su<v_\varepsilon\quad \text{in}\; \RN\}.$$
It is evident that $t_\varepsilon>0$. Furthermore, letting $\beta=[\min_{K}\frac{v}{u}]\wedge 1$, we note that
$\beta u\leq v$ in $K$ and $\sL(\beta u) + f(x, \beta u)\geq 0$ in $\RN$. Thus by (A)
we get $\beta u\leq v$ in $\RN$, and therefore, we have $t_\varepsilon\geq \beta$. Let $x_\varepsilon\in\argmax_K \frac{u}{v}$. Then, by \eqref{ET3.2B},
$$t_\varepsilon u(x_\varepsilon)\leq v_\varepsilon(x_\varepsilon)\Rightarrow t^{-1}_\varepsilon\geq 
\frac{u(x_\varepsilon)}{v_\varepsilon(x_\varepsilon)}\geq 1+\delta,$$
implying 
$$t_\varepsilon\leq 1-\delta_0, \quad \delta_0=\frac{\delta}{1+\delta}\,,$$
for all $\varepsilon\leq \varepsilon_0$. We claim that there exists a $\delta_1>0$ such that
for any $\varepsilon\leq \varepsilon_0$, we have
\begin{equation}\label{ET3.2C}
t^{3-\gamma}_\varepsilon f(x, u(x))+\delta_1\leq  f(x, t_\varepsilon u(x))\quad \text{for all}\; x\in K.
\end{equation}
In fact, \eqref{ET3.2C} follows from the following inequality
\begin{equation}\label{ET3.2D}
\min_{s\in [\beta, 1-\delta_0]} \min_{x\in K}\left(\frac{f(x, su(x))}{(su(x))^{3-\gamma}}
-\frac{f(x, u(x))}{(u(x))^{3-\gamma}}\right)>0\,,
\end{equation}
and the fact $\min_K u>0$. \eqref{ET3.2D} is a consequence of continuity and strict monotonicity of 
$\frac{f(x, s)}{s^{3-\gamma}}$ in $s$.
 Now choose $\varepsilon\in (0, \varepsilon_0)$ small enough so that
\begin{equation}\label{ET3.2E}
f(x, v_\varepsilon) < \delta_1/2 + f(x, v)\quad \text{for}\; x\in K.
\end{equation}
Also, by \eqref{ET3.2A0}, $f(x, v_\varepsilon)\leq f(x, v)$ for $x\in K^c$.
Now do the coupling between $v_\varepsilon$ and $\bar{u}=t_\varepsilon u$, as done in (A),
and we arrive at the following
relation
$$-t_\varepsilon^{3-\gamma} f(z, u(z))\leq -f(z, v(z))\leq - f(z, v_\varepsilon(z)) + \frac{\delta_1}{2}\Ind_K(z),$$
and $\bar{u}(z)=v_\varepsilon(z)$ for some $z\in\RN$, where the right most inequality follows from
\eqref{ET3.2E}.
 If $z\in K^c$ then contradiction follows from strict monotonicity, as in (A).
If $z\in K$ then we know from \eqref{ET3.2C}
$$-t_\varepsilon^{3-\gamma} f(z, u(z))\geq - f(z, \bar{u}(z)) + \delta_1= - f(z, v_\varepsilon(z)) + \delta_1,$$
which again leads to a contradiction. Hence the proof.
\end{proof}

Next we prove a decay estimate at the infinity.
\begin{lemma}\label{L3.1}
Assume (F3) and $|H(x, q)|\leq C |q|^{3-\gamma}$.
Let $u\geq 0$ be a bounded solution to $\sL u + f(x, u)\geq 0$ in $\RN$. Then we have 
$\lim_{|x|\to\infty}u=0$. 
\end{lemma}

\begin{proof}
Since $f(x, s)\leq s^{3-\gamma} a(x)$ by monotonicity, it holds that
$$\sL u + a(x) u^{3-\gamma}\geq 0 \quad \text{in}\; \RN.$$
Also, $\limsup_{|x|\to\infty} a(x)<0$. Then the proof follows from the arguments in Theorem~\ref{T1.3}(ii).
\end{proof}

Now we complete the proof of Theorem~\ref{T1.8}
\begin{proof}[Proof of Theorem~\ref{T1.8}]
(i)\, Let $\plam(\sL + a)>0$. Suppose that $u\gneq 0$ be a bounded solution of 
$$\sL + f(x, u)=0\quad \text{in}\; \RN.$$
Since $\frac{f(x, u(x))}{u^{3-\gamma}(x)}$ is finite, it follows from the strong maximum principle 
\cite[Theorem~2.2]{BV20} that $u>0$ in $\RN$. From Lemma~\ref{L3.1} we get that $u\in\cC^+_0(\RN)$.
Now, using monotonicity of $\frac{f(x, s)}{s^{3-\gamma}}$ we note that 
$$\sL u + a(x) u^{3-\gamma}\geq 0 \quad \text{in}\; \RN,$$
implying $\plam(\sL+a)\leq 0$. But this is a contradiction to the fact that $\plam(\sL + a)>0$. Hence
\eqref{EF3} can not have any nonnegative solution other that $u=0$.

(ii)\,  
By Lemma~\ref{L3.1} we know that any bounded solution belongs to $\cC^+_0(\RN)$. Therefore,
uniqueness part follows from Theorem~\ref{T3.2}. It is then enough
to prove the existence of a bounded solution. Let $\plam(\sL + a)<0$. Choose $\delta>0$ small enough so that
there exists $\psi\in\cC^+_0(\RN)$ satisfying
\begin{equation}\label{ET1.7B}
\sL\psi + (a(x)-2\delta)\psi^{3-\gamma}\geq 0\quad \text{in}\; \RN.
\end{equation}
Normalizing $\psi$ we may assume that $\norm{\psi}_\infty\leq M$.
Denote $\psi_\varepsilon=\psi-\varepsilon$ and $\cO_\varepsilon=\{\psi>\varepsilon\}$. Choose $\varepsilon$ 
small enough so that $K\Subset \cO_\varepsilon$ where $K$ is a compact set satisfying $a(x)<0$ in $K^c$.
Using convexity  we note that 
$$(\psi-\varepsilon)^{3-\gamma}\geq \psi^{3-\gamma}(x) - (3-\gamma) \psi^{2-\gamma}(x) \varepsilon
\geq \psi^{3-\gamma}(x) - (3-\gamma) M^{2-\gamma}\varepsilon\quad \text{in}\; \cO_\varepsilon.$$
Thus, in $\cO_\varepsilon$,
\begin{align*}
(a(x)-2\delta)\psi^{3-\gamma}_\varepsilon(x) &\geq (a(x)-2\delta)\psi^{3-\gamma}(x)
- (a(x)-2\delta)_+ (3-\gamma) M^{2-\gamma}\varepsilon
\\
&\geq(a(x)-2\delta)\psi^{3-\gamma}(x)
- \max_{K}|a(x)|\Ind_{K}(x)\, (3-\gamma) M^{2-\gamma}\varepsilon
\\
&\geq(a(x)-2\delta)\psi^{3-\gamma}(x)
- \max_{K}|a(x)|[\min_{K}\psi_\varepsilon]^{\gamma-3}\psi_{\varepsilon}^{3-\gamma}(x)
\, (3-\gamma) M^{2-\gamma}\varepsilon.
\end{align*}
Choosing $\varepsilon$ small enough we obtain from \eqref{ET1.7B} that
\begin{equation}\label{ET1.7C}
\sL\psi_\varepsilon + (a(x)-\delta)\psi_\varepsilon^{3-\gamma}\geq 0\quad \text{in}\; \cO_\varepsilon.
\end{equation}
Fix this choice of $\varepsilon$. Then we have $\cO_\varepsilon$ bounded and $\psi_\varepsilon=0$ on
$\partial\cO_\varepsilon$. Also, due to scale invariance, we note that \eqref{ET1.7C} also holds for 
any $\kappa\psi_\varepsilon$ for $\kappa>0$. Using (F3), we can find $\kappa$ small such that
$$f(x, s)\geq (a(x)-\delta)s^{3-\gamma}\quad \text{for}\; s\leq \kappa M.$$
This gives a subsolution from \eqref{ET1.7C} i.e.,
$$\sL(\kappa\psi_\varepsilon) + f(x, \kappa\psi_\varepsilon)\geq 0\quad \text{in}\; \cO_\varepsilon.$$
Also, $M$ is a supersolution. It is now standard to find a positive, bounded solution of \eqref{ET1.7A}
using monotone iteration method (see for instance, \cite[Lemma~4.1]{BV20}).
\end{proof}

\section{A boundary Harnack inequality}\label{S-BHI}

In this section we prove a boundary Harnack property. To describe the result we need a few notations. We assume that
$\cO$ is a $\cC^2$ domain. For $x\in\partial\cO$ we denote by $\cO_r(x)=\sB_r(x)\cap\cO$. By
$\nu_x$ we denote the inward normal at $x$. Suppose that $0\in\partial\cO$.
Let $\delta>0$ be such that any point $x\in\partial\cO\cap\sB_{4\delta}(0)$ satisfying a
inner sphere condition at $x$ with radius $2\delta$.
For $x\in\partial\cO$, we define $x_\delta=x+\delta \nu_x$.

Next we prove a boundary Harnack property.
\begin{theorem}[Boundary Harnack inequality]\label{T2.2}
There exists a constant $C$, dependent on $\delta, \cO, \theta, \mu, \gamma, N$, such that for
any positive $u$, vanishing continuously on $\sB_{4\delta}(0)\cap\partial\cO$ and satisfying
\begin{equation}\label{ET2.2A}
\ginfdel u - \theta |\grad u|^{3-\gamma} - \mu u^{3-\gamma}
\leq\, 0\,\leq\, \ginfdel u + \theta |\grad u|^{3-\gamma} + \mu u^{3-\gamma}
\quad \text{in}\; \cO_{4\delta}(0),
\end{equation}
we have
\begin{equation}\label{ET2.2B}
\sup_{\cO_\delta(0)} u \,\leq\, C u(0_\delta).
\end{equation}
\end{theorem}

\begin{proof}
We follow the idea of \cite[Theorem~1.1]{TB07} which deals with positive infinite harmonic functions.
Due to the availability of interior Harnack inequality, Theorem~\ref{T2.1}, it is enough to 
prove \eqref{ET2.2B} for some small $\delta$.
Denote by $\Omega_\delta=\sB_{4\delta}\cap\{y\in\cO\;:\; \dist(y, \partial\cO)\geq \delta\}$.
Choose $\delta$ small enough so that for $x\in\sB_{4\delta}$ there is a unique point on $\bar{x}\in\partial\cO$ such that $x=\bar{x}+d(x) \nu_{\bar x}$ where $d(x)=\dist(x, \partial\cO)$.
We can even choose it smaller so that for $v_1(x)=(2\delta)^{1/2}-|x|^{1/2}$ and $v_2=|x|^{1/2}$
 we have 
$$ \ginfdel v_1 - \theta |\grad v_1|^{3-\gamma} - \mu v_1 > \delta_1,
\;\text{and}\; \ginfdel v_2 + \theta |\grad v_2|^{3-\gamma} + \mu (2\delta)^{\frac{3-\gamma}{2}}<-\delta_1
 \quad \text{for}\; 0<|x|\leq 2\delta,$$
for some $\delta_1>0$.
By interior Harnack inequality, Theorem~\ref{T2.1}, there exists $M_1>0$, dependent on $\delta$, satisfying 
\begin{equation}\label{ET2.2C}
\sup_{\Omega_\delta} u \leq M_1 u(0_\delta).
\end{equation}
On the other hand, from the proof of \eqref{ET2.1C}
there exists a universal constant $M_2$ such that for any $x\in\cO_{4\delta}(0)$ with
$d(x)<2\delta$ we have 
$$ u(x)\leq M_2 u(\bar{x} + \frac{3}{2}d(x)\nu_{\bar x}).$$
Let $M=\max\{M_1,M_2\}$. Then
the above conclusion gives us 
\begin{equation}\label{ET2.2D}
u(x)\leq M u(\bar{x} + 2d(x)\nu_{\bar x}) \quad \text{for}\; d(x)<\delta.
\end{equation}
Repeating \eqref{ET2.2D} we note that, if $\delta/2^k\leq d(x)<\frac{\delta}{2^{k-1}} $, 
\begin{equation}\label{ET2.2E}
u(x)\leq M^{k-1} u(\bar{x} + 2^{k-1}d(x)\nu_{\bar x})\leq M^k u(0_z),
\end{equation}
where in last line we used \eqref{ET2.2C}.
To establish \eqref{ET2.2B} it is then enough to show that there exists $l$ large, not depending on $u$, such that if 
$d(\xi)<\frac{\delta}{2^{l}}$, $\xi\in\cO_\delta$, we have 
\begin{equation}\label{ET2.2F}
u(\xi)\leq M^{l+3} u(0_z).
\end{equation}
To establish \eqref{ET2.2F} we need an oscillation estimate of $u$. Take $z\in\sB_{3\delta}\cap \partial\cO$
and define $V_r(z)=\sup_{y\in \cO_r(z)} u(y)$. We claim that, for $r< \delta$,
\begin{equation}\label{ET2.2H}
V_s(z)\leq V_r(z) \left[\frac{s}{r}\right]^{1/2}\quad \text{for}\; 0<s\leq r. 
\end{equation}
Since $u$ is positive in $\cO$ it is evident that $V_s(z)>0$ for all $s>0$. Now fix $s\in (0, r)$. Since
$\lim_{s\to 0} V_s(z)=0$ by continuity, we can find $\kappa\in (0, s)$ satisfying $V_\kappa(z)<V_r(z)$.
Let 
$$w(y) = V_\kappa(z)+ [V_r(z)-V_\kappa(z)]
\frac{|y-z|^{\frac{1}{2}} - \kappa^{\frac{1}{2}}}{r^{\frac{1}{2}}-\kappa^{\frac{1}{2}}}, \quad y\in\cO_r(z)\setminus\bar\cO_\kappa(z).$$
By our choice of $\delta$, we have
\begin{equation}\label{ET2.2G}
\ginfdel w + \theta |\grad w|^{3-\gamma} + \mu w^{3-\gamma} < -\delta_1\left[\frac{V_r(z)-V_\kappa(z)}{r^{\frac{1}{2}}-\kappa^{\frac{1}{2}}}\right]^{3-\gamma}+ \mu L\, V_{\kappa}(z)
\quad \text{in}\; 
\cO_r(z)\setminus\bar\cO_\kappa(z) ,
\end{equation}
for some constant $L$, dependent on $V_r(z)$. Choosing $\kappa$ further small we obtain from
\eqref{ET2.2G} that
\begin{equation}\label{ET2.2G1}
\ginfdel w + \theta |\grad w|^{3-\gamma} + \mu w^{3-\gamma} < -\delta_2
\quad \text{in}\; 
\cO_r(z)\setminus\bar\cO_\kappa(z) ,
\end{equation}
for some $\delta_2>0$.
Since $\mu\geq 0$ we can not directly
apply the  comparison principle \cite[Theorem~2.1]{BV20} to conclude $u\leq w$ in $\cO_r(z)\setminus\bar\cO_\kappa(z)$. To apply comparison principle
we first note that $u\leq w$ on $\partial(\cO_r(z)\setminus\bar\cO_\kappa(z))$
and $w>0$ on $\partial(\cO_r(z)\setminus\bar\cO_\kappa(z))$.
Let $t=\max\{s\geq 0\; :\; su <w \; \text{in}\; \cO_r(z)\setminus\bar\cO_\kappa(z)\}$.
It is easily seen that $t>0$. If we can show that $t\geq 1$ then we would have $u\leq w$ in 
$\cO_r(z)\setminus\bar\cO_\kappa(z)$. On the contrary,
suppose that $t<1$ and let $w_1=t^{-1}w$. By \eqref{ET2.2G1} we have
$$\ginfdel w + \theta |\grad w|^{3-\gamma} + \mu w^{3-\gamma} < - t^{\gamma-3}\delta_2\quad \text{in}\; \cO_r(z)\setminus\bar\cO_\kappa(z),$$
and $M=\sup(u-w_1)=0$.
 Also, the value $M$ is attained inside $\cO_r(z)\setminus\bar\cO_\kappa(z)$ since $u\leq w< w_1$
 on $\partial(\cO_r(z)\setminus\bar\cO_\kappa(z))$.
Now repeat the proof of comparison principle in Theorem~\ref{T2.1} by considering coupling
function
between $u$ and $w_1$ to reach a 
contradiction. Thus we must have $u\leq w$ in $\cO_r(z)\setminus\bar\cO_\kappa(z)$.
Now letting $\kappa\to 0$ we obtain
$$u(y)\leq V_r(z)
\frac{|y-z|^{\frac{1}{2}}}{r^{\frac{1}{2}}}\quad \text{for}\;y\in\cO_r(z),$$
which in turn, gives \eqref{ET2.2H}.

Now suppose, on the contrary, that \eqref{ET2.2F} does not hold. Then there exists $\xi$ satisfying
$u(\xi)> M^{l+3} u(0_z)$. Let $p_0$ be the point on the boundary satisfying $d(\xi)=|\xi-p_0|$.
Applying \eqref{ET2.2H} we then have
$$V_{\delta 2^{-l+m}}(p_0) \geq 2^{m/2} V_{\delta 2^{-l}}(0)
\geq 2^{m/2} u(\xi)\geq 2^{m/2} M^{l+3} u(0_\delta).$$
Choose $m$ large to satisfy $2^{m/2}>M^2$, giving us $V_{\delta 2^{-l+m}}(p_0)\geq M^{l+5} u(0_\delta)$.
Choose a point $\xi_1\in \bar\cO_{\delta 2^{-l+m}}(p_0)$ satisfying
$$u(\xi_1)\geq M^{l+5} u(o_\delta).$$
Note that $|\xi-\xi_1|\leq \delta 2^{-l+m+1}$, and by \eqref{ET2.2E}, $d(\xi_1)<\delta/2^{l+4}$. Now 
repeat the above argument (replacing $l$ by $l+4$) to find $\xi_2$ with
$$u(\xi_2)\geq M^{l+2.4+1} u(0_\delta), \quad |\xi_1-\xi_2|\leq \delta 2^{-l-4+m+1},\quad 
d(\xi_2)\leq \delta 2^{-l-2.4}.$$
Hence we get a sequence $\{\xi_k\}$ satisfying
$$u(\xi_k)\geq M^{l+k.4+1} u(0_\delta), \quad |\xi_{k-1}-\xi_k|\leq \delta 2^{-l-4(k-1)+m+1},\quad 
d(\xi_k)\leq \delta 2^{-l-k.4}.$$
Note that 
$$|\xi_k-0|\leq |\xi -0| + \sum_{k\geq 1} |\xi_k-\xi_{k-1}|\leq \delta (1+ 2^{-l+m+1}\sum_{k\geq 1} 2^{-4(k-1)}).$$
Thus, we can choose $l$ large, dependent on $m$, so that $|\xi_k-0|<2\delta$ for all $k$.
Note that $\xi_k$ converges to the boundary and this leads to a contradiction
together with $u(\xi_k)\geq M^{l+k.4+1} u(0_\delta)$ since $u$ vanishes on the boundary.
Thus we have \eqref{ET2.2F} with the above choice of $l$. This completes the proof.
\end{proof}

\subsection*{Acknowledgement}
The research of Anup Biswas was supported in part by DST-SERB grants EMR/2016/004810 and MTR/2018/000028.

\end{document}